\documentclass{article}
\usepackage[toc,page]{appendix}
\usepackage[utf8]{inputenc}
\usepackage[english]{babel}
\DeclareMathAlphabet{\mathscr}{OT1}{pzc}{m}{it} 
\usepackage{amsmath}
\usepackage{dsfont} 
\usepackage{euscript}
\usepackage{amssymb}
\usepackage[T1]{fontenc}
\usepackage{mathtools}   
\usepackage{amsfonts}
\usepackage{bbm}
\usepackage{stmaryrd}
\usepackage{mathrsfs}  
\usepackage{graphicx}
\usepackage{bigints}
\usepackage{relsize}
\usepackage[colorinlistoftodos]{todonotes}
\usepackage{amssymb,amsmath,amsthm}
\numberwithin{equation}{section}
\usepackage{dsfont}
\usepackage[shortlabels]{enumitem}
\setlist{labelindent=\parindent,leftmargin=*}
\usepackage{authblk}

\newtheorem{theorem}{Theorem}[section]
\newtheorem{notation}[theorem]{Notation}
\newtheorem{lemma}[theorem]{Lemma}
\newtheorem{proposition}[theorem]{Proposition}
\newtheorem{corollary}[theorem]{Corollary}
\newtheorem{definition}[theorem]{Definition}
\newtheorem{hypothesis}[theorem]{Hypothesis}
\newtheorem{remark}[theorem]{Remark}

\usepackage{amsmath,amssymb}
\makeatletter
\newsavebox\myboxA
\newsavebox\myboxB
\newlength\mylenA
\newcommand*\xoverline[2][0.75]{%
    \sbox{\myboxA}{$\m@th#2$}%
    \setbox\myboxB\null
    \ht\myboxB=\ht\myboxA%
    \dp\myboxB=\dp\myboxA%
    \wd\myboxB=#1\wd\myboxA
    \sbox\myboxB{$\m@th\overline{\copy\myboxB}$}
    \setlength\mylenA{\the\wd\myboxA}
    \addtolength\mylenA{-\the\wd\myboxB}%
    \ifdim\wd\myboxB<\wd\myboxA%
       \rlap{\hskip 0.5\mylenA\usebox\myboxB}{\usebox\myboxA}%
    \else
        \hskip -0.5\mylenA\rlap{\usebox\myboxA}{\hskip 0.5\mylenA\usebox\myboxB}%
    \fi}
\makeatother

\title{BSDEs with no driving martingale, Markov processes and 
associated Pseudo Partial Differential Equations. Part II: Decoupled mild solutions and Examples}

\author{
  Adrien BARRASSO \thanks{Université d'Évry Val d'Essonne, 
  Laboratoire de Mathématiques et Modélisation,
23 Bd. de France, 
91037 Évry Cedex,  
 E-mail: \sf adrien.barrasso@univ-evry.fr}  
	\qquad\quad
	Francesco RUSSO\thanks{ENSTA Paris, Institut Polytechnique
          de Paris, Unit\'e de Math\'ematiques appliqu\'ees, 828, boulevard des Mar\'echaux, F-91120 Palaiseau, France. E-mail:
          \sf russo@math.univ-paris13.fr}}

\date{March 2021}

\begin{document}
	\maketitle
	
{\bf Abstract.} Let 
$(\mathbbm{P}^{s,x})_{(s,x)\in[0,T]\times E}$ be a family of probability measures,
 where $E$ is a Polish space,
defined on the canonical probability space ${\mathbbm D}([0,T],E)$
of $E$-valued c\`adl\`ag functions. We suppose that a martingale problem with 
respect to a time-inhomogeneous generator $a$  is well-posed.
We  consider also an associated  semilinear {\it Pseudo-PDE}
for which we introduce a notion of so called {\it decoupled mild} solution
 and study the equivalence with the
notion of martingale solution introduced in a companion paper.
We also investigate well-posedness for decoupled mild solutions and their
relations with a special class of BSDEs without driving martingale.
The notion of decoupled mild solution is a good candidate to replace the
notion of viscosity solution which is not always suitable
when the map $a$ is not a PDE operator.

	\bigskip
	{\bf MSC 2020} Classification. 
	60H30; 60H10; 35S05; 60J35; 60J60; 60J99.
	
	
	\section{Introduction}

	The framework of this paper is the canonical space 
 $\Omega = \mathbbm{D}([0,T],E)$ of c\`adl\`ag functions defined on 
the interval $[0,T]$
with values in a Polish space $E$. This space will be equipped 
with a family $(\mathbbm{P}^{s,x})_{(s,x)\in[0,T]\times E}$ of probability measures 
indexed by an initial time $s\in[0,T]$ and a starting point $x \in E$.
For each $(s,x)\in[0,T]\times E$, $\mathbbm{P}^{s,x}$ corresponds to
 the law of an underlying forward Markov process  with time index $[0,T]$, taking values in the  Polish state space $E$  which is
	characterized as the solution of a  well-posed martingale problem
	related to a certain operator $(\mathcal{D}(a),a)$.
In the companion paper 
\cite{paper1preprint}
we have introduced a semilinear equation generated by $(\mathcal{D}(a),a)$, called 
{\it Pseudo-PDE} of the type 
	\begin{equation}\label{PDEIntro}
	\left\{
	\begin{array}{rccc}
	a(u) + f\left(\cdot,\cdot,u,\Gamma(u)^{\frac{1}{2}}\right)&=&0& \text{ on } [0,T]\times E   \\
	u(T,\cdot)&=&g,& 
	\end{array}\right.
	\end{equation}
	where $\Gamma(u)=a(u^2)-2ua(u)$ is a potential theory operator called
	the {\it carr\'e du champs operator}. 
A classical solution of \eqref{PDEIntro} is defined as an element of
 $\mathcal{D}(a)$ verifying  \eqref{PDEIntro}. In 
\cite{paper1preprint}
we have also defined the notion of {\it martingale solution}
of \eqref{PDEIntro}, see  Definition \ref{D417}. 
A function $u$ is a martingale solution if \eqref{PDEIntro}
holds replacing the map $a$ (resp. $\Gamma$) with an extended
operator $\mathfrak{a}$ (resp. $\mathfrak{G}$) which is introduced 
in Definition \ref{extended} (resp. \ref{extendedgamma}). 
The martingale solution extends the (analytical) notion
of classical solution, however it is a probabilistic concept.
The objectives of the present paper are essentially three.
\begin{enumerate}
\item To introduce an alternative notion of (this time analytical) solution, 
that we call {\it decoupled mild}, 
since it makes use of the time-dependent transition kernel 
associated with $a$. 
This new type of solution will be shown to be  essentially equivalent to
the  martingale one. 
\item To show existence and uniqueness of decoupled mild solutions.
\item To emphasize the link with solutions of forward BSDEs (FBSDEs)
without driving martingale introduced in 
\cite{paper1preprint}.
\end{enumerate}
The aforementioned  FBSDEs are of the  form
\begin{equation}\label{BSDEIntro}
	Y^{s,x}_t = g(X_T) + \int_t^T f\left(r,X_r,Y^{s,x}_r,\sqrt{\frac{d\langle M^{s,x}\rangle_r}{dr}}\right)dr  -(M^{s,x}_T - M^{s,x}_t),
\end{equation}
in  a stochastic basis $\left(\Omega,\mathcal{F}^{s,x},(\mathcal{F}^{s,x}_t)_{t\in[0,T]},\mathbbm{P}^{s,x}\right)$ which depends on $(s,x)$.
Under suitable conditions, for fixed $(s,x)$, the solution of
this FBSDE is a couple $(Y^{s,x},M^{s,x})$
of c\`adl\`ag stochastic processes where $M^{s,x}$ is a martingale. 
This was introduced and studied in a more general setting in \cite{paper1preprint}, see \cite{qianBSDEs} for a similar formulation.

We refer to the introduction  and reference list of previous paper 
for an extensive description of contributions to non-Brownian type BSDEs.
The classical forward BSDE, which is driven by a Brownian motion
is of the form 
\begin{equation}\label{BrownianBSDE}
\left\{\begin{array}{rcl}
X^{s,x}_t &=& x+ \int_s^t \mu(r,X^{s,x}_r)dr +\int_s^t \sigma(r,X^{s,x}_r)dB_r\\
Y^{s,x}_t &=& g(X^{s,x}_T) + \int_t^T f\left(r,X^{s,x}_r,Y^{s,x}_r,Z^{s,x}_r\right)dr  -\int_t^T Z^{s,x}_rdB_r,
\end{array}\right.
\end{equation}
where $B$ is a Brownian motion. 
Existence and uniqueness for \eqref{BrownianBSDE}
 was established first
supposing mainly Lipschitz conditions on $f$ with respect to the third
and fourth variable. $\mu$ and $\sigma$ were also assumed to be 
Lipschitz (with respect to $x$) and to have linear growth.
In the sequel those conditions were
considerably relaxed, see \cite{PardouxRascanu}
 and references therein.
	This is a particular case of a more general (non-Markovian) Brownian BSDE 
introduced  in 1990 by E. Pardoux and
	S. Peng in \cite{parpen90}, after an early work
	of J.M. Bismut in 1973 in \cite{bismut}.

Equation \eqref{BrownianBSDE} is a probabilistic representation of 
a  semilinear partial differential
equation of parabolic type with terminal condition:
\begin{equation}\label{PDEparabolique}
\left\{
\begin{array}{l}
\partial_tu + \frac{1}{2}\underset{i,j\leq d}{\sum} (\sigma\sigma^\intercal)_{i,j}\partial^2_{x_ix_j}u + \underset{i\leq d}{\sum} \mu_i\partial_{x_i}u + f(\cdot,\cdot,u,\sigma\nabla u)=0\quad \text{ on } [0,T[\times\mathbbm{R}^d \\
u(T,\cdot) = g. 
\end{array}\right.
\end{equation}

Given, for every $(s,x)$, a solution $(Y^{s,x}, Z^{s,x})$ of the FBSDE \eqref{BrownianBSDE},
 under some continuity assumptions on the coefficients, see e.g. \cite{pardoux1992backward}, it was proved that
the function $u(s,x):= Y^{s,x}_s$ is 
a viscosity solution of  \eqref{PDEparabolique},
see also  \cite{peng1991probabilistic, 
pardoux1992backward, peng1991probabilistic, el1997backward}, for related work.

We prolong this idea in a general case where the FBSDE is
\eqref{BSDEIntro} with solution $(Y^{s,x}, M^{s,x})$.
In that case $u(s,x):= Y^{s,x}_s$ will be the decoupled mild solution of
\eqref{PDEIntro}, see Theorem \ref{Representation};
 in that general context the decoupled mild solution replaces
the one of viscosity, for reasons that we will explain below.
One celebrated problem in the case of Brownian FBSDEs is the characterization
of $Z^{s,x}$ through a deterministic function $v$.
This is what we will call the {\it identification  problem}.
In general the link between
$v$ and $u$ is not always analytically established, excepted when $u$
has some suitable differentiability property, see e.g. 
\cite{barles1997sde}: 
in that case  $v$ is closely related to the gradient of $u$.
In our case, the notion of decoupled mild solution allows to identify $(u,v)$
as the analytical solution of a deterministic problem.
  In the literature, the notion of mild 
	solution of  PDEs  was used  in finite dimension in 
\cite{BSDEmildPardouxBally},
	where the authors tackled diffusion operators generating symmetric Dirichlet forms and associated Markov processes thanks to the theory of Fukushima Dirichlet forms, see e.g. \cite{fuosta}. A partial extension to
the case of non-symmetric Dirichlet forms is performed in \cite{klimsiak}.
	Infinite dimensional setups were considered  for example  
	in  \cite{fuhrman_tessitore02} where an infinite
	dimensional BSDE  could produce the mild solution  of a PDE on a 
	Hilbert space. 

Let $B$ be a functional Banach space $(B,\|\cdot\|)$
of real Borel functions defined on $E$ 
 and 
 $A$ be an unbounded operator on  $(B,\|\cdot\|)$. 
In the theory of evolution equations one often considers systems of
the type  
\begin{equation}\label{linear}
\left\{\begin{array}{rcl}
\partial_tu+Au  &=& l \text{ on }[0,T]\times\mathbbm{R}^d\\
u(T,\cdot) &=& g,
\end{array}\right.
\end{equation}
where 
 $l:[0,T] \times \mathbbm{R}^d \rightarrow {\mathbbm R}$
and $g:\mathbbm{R}^d \rightarrow {\mathbbm R}$ are 
such that $l(t,\cdot)$ and $g$ belong to $B$ for every
$t \in [0,T]$. 
 The idea of mild solutions 
consists to consider $A$ (when possible) as the infinitesimal generator
of a semigroup of operators
$(P_t)_{t\geq 0}$ on $(B,\|\cdot\|)$,  
in the following sense. There is $\mathcal{D}(A) \subset B$,
 a dense subset on which
 $A f=\underset{t\rightarrow 0^+}{\text{lim }}\frac{1}{t}(P_tf-f)$.
In particular one may think of $(P_t)_{t\geq 0}$ as the heat kernel semi-group and
$A$ as $\frac{1}{2}\Delta $ which is the infinitesimal generator of the Brownian motion. 
The approach of mild solutions is also very popular  in the framework of 
stochastic PDEs see e. g. 
\cite{daprato_zabczyk14}.
When $A$ is a local operator, one solution (in the sense of distributions, or in the sense of 
evaluation against test functions)
to the linear evolution problem with terminal condition
\eqref{linear} is
the so called {\it mild solution} 
 \begin{equation}\label{mildlinear} 
 u(s,\cdot)= P_{T-s}[g] - \int_s^T P_{r-s}[l(r,\cdot)]dr.
 \end{equation}
If $ l$ is explicitly a function of $u$ then
\eqref{mildlinear} becomes itself an equation
and a mild solution would consist in finding a fixed
point of \eqref{mildlinear}.
Let us now suppose the existence of a  
 map
$S: \mathcal{D}(S) \subset B \rightarrow B$, typically  $S$ being the gradient,
when $(P_t)$ is the heat kernel semigroup.
The natural question is what would be a natural replacement for
a {\it mild solution}
for
 \begin{equation}\label{linearS}
 \left\{\begin{array}{rcl}
 \partial_tu+Au  &=& - f(s,\cdot,u,Su) \text{ on }[0,T]\times\mathbbm{R}^d\\
 u(T,\cdot) &=& g.
 \end{array}\right.
 \end{equation}
If the domain of $S$ is $B$, then it is not
difficult to extend the notion of mild solution to this case.
One novelty of our approach consists is considering
the case of solutions $u:[0,T] \times \mathbbm{R}^d \rightarrow \mathbbm{R}$
for which $Su(t,\cdot)$ is not defined.
\begin{enumerate} 
\item Suppose  one expects a solution not to be classical,
i.e. such that $u(r,\cdot)$ should 
not belong to the domain of $\mathcal{D}(A)$ but to be
in the domain of $S$.
In the case of usual PDEs,
 one thinks of  possible solutions which are not $C^{1,2}$
but admitting a gradient, typically viscosity solutions which are 
differentiable in $x$.
In that case the usual idea of
 mild solutions theory applies to equations of type
\eqref{linearS}.
In this setup, inspired by \eqref{mildlinear} a mild solution of the equation is naturally
defined as a solution of the integral equation 
\begin{equation}\label{E39}
u(s,\cdot)= P_{T-s}[g] + \int_s^T P_{r-s}[f(r,\cdot,u(r,\cdot),Su(r,\cdot))]dr.
\end{equation}
\item However, there may be reasons for which
the candidate solution $u$ is such that
$u(t,\cdot)$  does not even belong to $\mathcal{D}(S)$.
In the case of PDEs it is often the case for viscosity solutions of  PDEs which do not admit a gradient. 
In that case the idea is to replace \eqref{E39}
with 
\begin{equation}
u(s,\cdot)= P_{T-s}[g] + \int_s^T P_{r-s}[f(r,\cdot,u(r,\cdot),v(r,\cdot))]dr.
\end{equation}
and to add a second equality which expresses in a {\it mild} form
the equality $v(r,\cdot) = Su(r,\cdot)$.
\end{enumerate}
We will work out previous methodology for 
the $Pseudo-PDE(f,g)$. In that case
 $S$ will be given by the mapping 
$u\longmapsto \Gamma(u)^{\frac{1}{2}}$.
If  $A=\frac{1}{2}\Delta$ for instance one would have $\Gamma(u)^{\frac{1}{2}}=\|\nabla u\|$.
For pedagogical purposes, one can first consider an operator $a$ of
 type $\partial_t+A$ when $A$ is the generator
of a Markovian (time-homogeneous)   semigroup. 
 In this case, 
 \begin{eqnarray*}
 \Gamma(u) &=& \partial_t(u^2)+A(u^2)-2u\partial_tu-2uAu \\
 &=&  A(u^2) - 2uAu.
 \end{eqnarray*}
 Equation 
\begin{equation}\label{Emilddiffusion}
\partial_tu+Au + f\left(\cdot,\cdot,u,\Gamma(u)^{\frac{1}{2}}\right)=0,
\end{equation} 
could therefore be decoupled into the system
\begin{equation}
\left\{
\begin{array}{l}
\partial_tu+Au + f(\cdot,\cdot,u,v)= 0 \\
v^2=\partial_t(u^2)+A(u^2)-2u(\partial_tu+Au),
\end{array}\right.
\end{equation}
which furthermore can be expressed as
\begin{equation}
\left\{
\begin{array}{rcl}
\partial_tu+Au &=& - f(\cdot,\cdot,u,v)\\
\partial_t(u^2)+A(u^2)&=&v^2-2uf(\cdot,\cdot,u,v).
\end{array}\right.
\end{equation}
Taking into account the existing notions of mild solution \eqref{mildlinear} (resp. \eqref{E39}), for corresponding equations \eqref{linear} (resp. \eqref{linearS}), one is naturally tempted to define a decoupled mild solution of \eqref{PDEIntro} as a function $u$ for which there exist $v\geq 0$ such that 
\begin{equation}
\left\{
\begin{array}{rcl}
u(s,\cdot)&=& P_{T-s}[g] + \int_s^T P_{r-s}[f(r,\cdot,u(r,\cdot),v(r,\cdot))]dr\\
u^2(s,\cdot)&=& P_{T-s}[g^2] - \int_s^T P_{r-s}[v^2(r,\cdot)-2u(r,\cdot)f(r,\cdot,u(r,\cdot),v(r,\cdot))]dr.
\end{array}\right.
\end{equation}
As we mentioned before, our
 approach is alternative to a possible notion of viscosity
solution for the $Pseudo-PDE(f,g)$. That notion will be the object  of 
a subsequent paper, at least in the case when the driver do not depend
on the last variable. In the general case the notion of viscosity solution
does not fit well because of lack of suitable comparison theorems.
On the other hand, even in the recent literature (see \cite{barles1997backward}) in order to show existence of viscosity solutions
specific conditions exist on the driver.
In our opinion our approach of decoupled mild solutions for $Pseudo-PDE(f,g)$
constitutes an interesting novelty even in the case of 
semilinear parabolic PDEs.

The main contributions  of the paper are essentially the following.
In Section \ref{mild}, Definition \ref{mildsoluv} introduces our notion of decoupled mild solution of \eqref{PDEIntro} in the general setup.  In section Section \ref{mart-mild}, Proposition \ref{MartingaleImpliesMild} states that under a square integrability type condition, every martingale solution is a decoupled mild solution of \eqref{PDEIntro}. Conversely, Proposition \ref{MildImpliesMartingale} shows that every decoupled mild solution is a martingale solution. In Theorem \ref{MainTheorem} we prove existence and uniqueness of a decoupled mild solution for \eqref{PDEIntro}. In Section \ref{FBSDEs}, we show how the unique decoupled mild solution of  \eqref{PDEIntro} can be represented via the FBSDEs \eqref{BSDEIntro}.
In Section \ref{exemples} we develop examples of Markov processes and corresponding operators $a$ falling into our abstract setup. In Section \ref{S4a}, we work in the setup of \cite{stroock}, the Markov process is a diffusion with jumps and the corresponding operator is of diffusion type with an additional non-local operator. In Section \ref{S4b} we consider Markov processes associated to pseudo-differential operators (typically the fractional Laplacian) as in \cite{jacob2005pseudo}. In Section \ref{S4c} we study a semilinear parabolic PDE with distributional drift, and the corresponding process is the solution an SDE with distributional drift as defined in \cite{frw1}. Finally in Section \ref{S4d} are interested with diffusions on differential manifolds and associated diffusion operators, an example being the Brownian motion in a Riemannian manifold associated to the Laplace-Beltrami operator.

\section{Preliminaries}\label{S1}

In this section we will recall the notations, notions and results of the companion paper 
\cite{paper1preprint},
 which will be used here.
\begin{notation}
In the whole paper, concerning functional spaces we will use the following notations.

A topological space $E$ will always be considered as a measurable space with its Borel $\sigma$-field which shall be denoted $\mathcal{B}(E)$. Given two topological spaces, $E,F$,  then $\mathcal{C}(E,F)$ (respectively  $\mathcal{B}(E,F)$)  will denote the set of functions from $E$ to $F$ which are continuous (respectively Borel) and if $F$ is a metric space, $\mathcal{C}_b(E,F)$ (respectively $\mathcal{B}_b(E,F)$) will denote the set of functions from $E$ to $F$ which are bounded continuous (respectively bounded Borel). For any $p\in[1,\infty]$, $d\in\mathbbm{N}^*$, $(L^p(\mathbbm{R}^d),\|\cdot\|_p)$ will denote the usual Lebesgue space equipped with its usual norm.
On a fixed probability space $\left(\Omega,\mathcal{F},\mathbbm{P}\right)$, for any $p\in\mathbbm{N}^*$, $L^p(\mathbbm{P})$ will denote the set of random variables (defined up to a.s equality) with finite $p$-th moment.
 A probability space equipped with a right-continuous filtration $\left(\Omega,\mathcal{F},(\mathcal{F}_t)_{t\in\mathbbm{T}},\mathbbm{P}\right)$ (where $\mathbbm{T}$ is equal to $\mathbbm{R}_+$ or to $[0,T]$ for some $T\in\mathbbm{R}_+^*$) will be called a \textbf{stochastic basis} and will be said to \textbf{fulfill the usual conditions} if the probability space is complete and if $\mathcal{F}_0$ contains all the $\mathbbm{P}$-negligible sets. When a stochastic basis is fixed, $\mathcal{P}ro$ denotes the \textbf{progressive $\sigma$-field} on $\mathbbm{T}\times\Omega$.

On a fixed stochastic basis $\left(\Omega,\mathcal{F},(\mathcal{F}_t)_{t\in\mathbbm{T}},\mathbbm{P}\right)$, we will use the following notations and vocabulary,
concerning spaces of stochastic processes, 
 most of them being taken or adapted from \cite{jacod79} or \cite{jacod}.
$\mathcal{M}$ will be the space of c\`adl\`ag martingales.  
For any $p\in[1,\infty]$  $\mathcal{H}^p$ will denote
the subset of $\mathcal{M}$ of elements $M$ such that $\underset{t\in \mathbbm{T}}{\text{sup }}|M_t|\in L^p(\mathbbm{P})$ and in this set we identify indistinguishable elements. It is a Banach space for  the norm
$\| M\|_{\mathcal{H}^p}=\mathbbm{E}[|\underset{t\in \mathbbm{T}}{\text{sup }}M_t|^p]^{\frac{1}{p}}$, and
$\mathcal{H}^p_0$ will denote the Banach subspace of $\mathcal{H}^p$
containing the elements starting at zero.
If $\mathbbm{T}=[0,T]$ for some $T\in\mathbbm{R}_+^*$, a stopping time will be considered as a random variable with values in 
$[0,T]\cup\{+\infty\}$.
We define a \textbf{localizing sequence of stopping times} as an increasing sequence of stopping times $(\tau_n)_{n\geq 0}$ such that there exists $N\in\mathbbm{N}$ for which $\tau_N=+\infty$. Let $Y$ be a process and $\tau$ a stopping time, we denote  $Y^{\tau}$ the process $t\mapsto Y_{t\wedge\tau}$ which we call \textbf{stopped process}.  If $\mathcal{C}$ is a set of processes, we define its \textbf{localized class} $\mathcal{C}_{loc}$ as the set of processes $Y$ such that there exist a localizing sequence $(\tau_n)_{n\geq 0}$ such that for every $n$, the stopped process $Y^{\tau_n}$ belongs to $\mathcal{C}$.
For any $M\in  \mathcal{M}_{loc}$, we denote $[M]$ its \textbf{quadratic variation} and if moreover  $M\in\mathcal{H}^2_{loc}$, $\langle M\rangle$ will denote its (predictable) \textbf{angular bracket}.
$\mathcal{H}_0^2$ will be equipped with scalar product defined by $(M,N)_{\mathcal{H}^2}=\mathbbm{E}[M_TN_T] 
=\mathbbm{E}[\langle M, N\rangle_T] $ which makes it a Hilbert space.
Two local martingales $M,N$ will be said to be \textbf{strongly orthogonal} if $MN$ is a local martingale starting in 0 at time 0. In $\mathcal{H}^2_{0,loc}$ this notion is equivalent to $\langle M,N\rangle=0$.

\end{notation}

As in previous paper 
\cite{paper1preprint}
we will be interested in
a Markov process which is 
the solution of a martingale problem which we now recall below.
For definitions and results concerning Markov processes, the reader may refer to Appendix \ref{A2}. In particular, let $E$ be a Polish space and $T\in\mathbbm{R}_+$ be a finite horizon we now consider  $\left(\Omega,\mathcal{F},(X_t)_{t\in[0,T]},(\mathcal{F}_t)_{t\in[0,T]}\right)$  the canonical space which was introduced in Notation \ref{canonicalspace}, and a Markov (canonical) class measurable in time $(\mathbbm{P}^{s,x})_{(s,x)\in[0,T]\times E}$, see  Definitions \ref{defMarkov} and \ref{DefFoncTrans}. We will also consider the completed stochastic basis $\left(\Omega,\mathcal{F}^{s,x},(\mathcal{F}^{s,x}_t)_{t\in[0,T]},\mathbbm{P}^{s,x}\right)$, 
see Definition \ref{CompletedBasis}.

We now recall what the notion of  martingale problem associated to
 an operator  introduced in Section 4 of \cite{paper1preprint}.
\begin{definition}\label{MartingaleProblem}
Given a linear algebra $\mathcal{D}(a)\subset \mathcal{B}([0,T]\times E,\mathbbm{R})$, and a linear operator $a$ mapping $\mathcal{D}(a)$ into $\mathcal{B}([0,T]\times E,\mathbbm{R})$, we say that a set of probability measures $(\mathbbm{P}^{s,x})_{(s,x)\in [0,T]\times E}$ defined on $(\Omega,\mathcal{F})$ solves the {\bf Martingale Problem associated to}  $(\mathcal{D}(a),a)$ if, for  any 
	$(s,x)\in[0,T]\times E$, $\mathbbm{P}^{s,x}$ verifies
	\begin{description}
		\item{(a)} $\mathbbm{P}^{s,x}(\forall t\in[0,s], X_t=x)=1$;
		\item{(b)} for every $\phi\in\mathcal{D}(a)$, the process $\phi(\cdot,X_{\cdot})-\int_s^{\cdot}a(\phi)(r,X_r)dr$, $t \in [s,T]$ 
		is  a c\`adl\`ag $(\mathbbm{P}^{s,x},(\mathcal{F}_t)_{t\in[s,T]})$-local martingale.
	\end{description}
	We say that the {\bf Martingale Problem is well-posed} if for any $(s,x)\in[0,T]\times E$, $\mathbbm{P}^{s,x}$ is the only probability measure 
satisfying the  properties (a) and (b).
\end{definition}

As for \cite{paper1preprint}, 
in the sequel of the paper we will assume the following.
\begin{hypothesis}\label{MPwellposed}
	The Markov canonical class $(\mathbbm{P}^{s,x})_{(s,x)\in [0,T]\times E}$ solves a well-posed Martingale Problem
 associated to some $(\mathcal{D}(a),a)$  in the sense 
	of  Definition \ref{MartingaleProblem}.
\end{hypothesis}

\begin{notation}\label{Mphi}
	For every $(s,x)\in[0,T]\times E$ and $\phi\in\mathcal{D}(a)$, the process 
	\\
	$t\mapsto\mathds{1}_{[s,T]}(t)\left(\phi(t,X_{t})-\phi(s,x)-\int_s^{t}a(\phi)(r,X_r)dr\right)$ will be denoted $M[\phi]^{s,x}$.
\end{notation}
$M[\phi]^{s,x}$ is a c\`adl\`ag $(\mathbbm{P}^{s,x},(\mathcal{F}_t)_{t\in[0,T]})$-local martingale equal to $0$ on $[0,s]$, and by Proposition \ref{ConditionalExp}, it is also a $(\mathbbm{P}^{s,x},(\mathcal{F}^{s,x}_t)_{t\in[0,T]})$-local martingale.

The bilinear operator below was introduced (in the case of time-homogeneous operators) by J.P. Roth in potential analysis (see Chapter III in \cite{roth}),
 and popularized by P.A. Meyer and others in the study of homogeneous Markov processes
(see for example Expos\'e II: L'op\'erateur carr\'e du champs in \cite{sem10} or 13.46 in \cite{jacod79}).
\begin{definition}\label{SFO}
We introduce the bilinear operator
\begin{equation}
\Gamma : \begin{array}{r c l}
    \mathcal{D}(a) \times \mathcal{D}(a)  & \rightarrow & \mathcal{B}([0,T]\times E) \\
    (\phi, \psi) & \mapsto & a(\phi\psi) - \phi a(\psi) - \psi a(\phi).
   \end{array}
\end{equation} 
The operator $\Gamma$ is 
called the \textbf{carr\'e du champs operator}.

$\Gamma(\phi,\phi)$ will more simply be denoted $\Gamma(\phi)$, and when this mapping takes positive values, $\Gamma(\phi)^{\frac{1}{2}}$ will denote its point-wise square root.
\end{definition}

The angular bracket of the martingales introduced in Notation \ref{Mphi}
 are expressed via the operator $\Gamma$. Proposition 4.8 of 
\cite{paper1preprint}, 
tells the following.

\begin{proposition}\label{bracketindomain}
For any $\phi\in \mathcal{D}(a)$ and $(s,x)\in[0,T]\times E$, $M[\phi]^{s,x}$ is in $\mathcal{H}^2_{0,loc}$.
Moreover, for any $(\phi, \psi)\in \mathcal{D}(a) \times \mathcal{D}(a)$ and $(s,x)\in[0,T]\times E$ we have in $(\Omega,\mathcal{F}^{s,x},(\mathcal{F}^{s,x}_t)_{t\in[0,T]},\mathbbm{P}^{s,x})$ and on the interval $[s,T]$
\begin{equation}
\langle M[\phi]^{s,x} , M[\psi]^{s,x} \rangle = \int_s^{\cdot} \Gamma(\phi,\psi)(r,X_r)dr.
\end{equation}
\end{proposition}

We introduce the space of square integrable martingales with absolutely continuous angle bracket.
\begin{notation} \label{D212}
$\mathcal{H}_0^{2,abs} := \{M\in\mathcal{H}^2_0|d\langle M\rangle_t \ll dt\}$. 
We will also denote $\mathcal{L}^2(dt\otimes d\mathbbm{P})$  the set of (up to indistinguishability)
 progressively measurable processes $\phi$ such that $\mathbbm{E}[\int_0^T \phi^2_rdr]<\infty$.
\end{notation}
We remark  $\mathcal{H}_0^{2,abs}$
corresponds in \cite{paper1preprint} (Section 3.) to
  $\mathcal{H}_0^{2,V}$. In this paper we have set $V_t \equiv t$.
Proposition 4.11 of 
\cite{paper1preprint}
 says the following.
\begin{proposition}\label{H2V=H2}
If Hypothesis \ref{MPwellposed} is verified then under any $\mathbbm{P}^{s,x}$, \\
$\mathcal{H}_0^2=\mathcal{H}_0^{2,abs}$.
\end{proposition}


In the sequel, several  functional equations  will 
hold up to a \textbf{zero potential} set
that we recall below.

\begin{definition}\label{zeropotential}
For any $(s,x)\in[0,T]\times E$ we define the \textbf{potential measure} $U(s,x,\cdot)$ on $\mathcal{B}([0,T]\times E)$ by
 $U(s,x,A) := \mathbbm{E}^{s,x}\left[\int_s^{T} \mathds{1}_{\{(t,X_t)\in A\}}dt\right]$.
\\
A Borel set $A\in\mathcal{B}([0,T]\times E)$ will be said to be
 {\bf of zero potential} if, for any $(s,x)\in[0,T]\times E$  we have  $U(s,x,A) = 0$.
\end{definition}

\begin{notation}\label{topo}
Let $p > 0$, we define
$$\mathcal{L}^p_{s,x} :=\left\{ f\in \mathcal{B}([0,T]\times E,\mathbbm{R}):\, \mathbbm{E}^{s,x}\left[\int_s^{T} |f|^p(r,X_r)dr\right] < \infty\right\},$$
on which we introduce the usual semi-norm $\|\cdot\|_{p,s,x}: f\mapsto\left(\mathbbm{E}^{s,x}\left[\int_s^{T}|f(r,X_r)|^pdr\right]\right)^{\frac{1}{p}}$
 We also denote 
$\mathcal{L}^0_{s,x}  :=\left\{ f\in \mathcal{B}([0,T]\times E,\mathbbm{R}):\, \int_s^{T} |f|(r,X_r)dr < \infty\,\mathbbm{P}^{s,x}\text{ a.s. }\right\}$.
\\
For any $p \ge0$, we then define an intersection of these spaces, 
i.e. 
\\
$\mathcal{L}^p_X:=\underset{(s,x)\in[0,T]\times E}{\bigcap}\mathcal{L}^p_{s,x}$.
Finally, let $\mathcal{N}$ the linear subspace of $\mathcal{B}([0,T]\times E,\mathbbm{R})$ containing all functions which are equal to 0 $U(s,x,\cdot)$ a.e. for every $(s,x)$. For any $p\in \mathbbm{N}$, we define  the quotient space $L^p_X := \mathcal{L}^p_X /\mathcal{N}$.
If $p\geq 1$, $L^p_X$ can be equipped with the topology generated by the family of semi-norms $\left(\|\cdot\|_{p,s,x}\right)_{(s,x)\in[0,T]\times E}$ which makes it into a separable locally convex topological vector space.
\end{notation}

The statement below was stated in Proposition 4.14 of 
\cite{paper1preprint}.
\begin{proposition}\label{uniquenessupto}
Let $f$ and $g$ be in $\mathcal{B}([0,T]\times E,\mathbbm{R})$ such that the processes $\int_s^{\cdot}f(r,X_r)dr$ and $\int_s^{\cdot}g(r,X_r)dr$ are finite $\mathbbm{P}^{s,x}$ a.s. for any $(s,x) \in [0,T] \times E$.
Then $f$ and $g$ are equal up a zero potential set if and only if $\int_s^{\cdot}f(r,X_r)dr$ and $\int_s^{\cdot}g(r,X_r)dr$ are indistinguishable under $\mathbbm{P}^{s,x}$ for any $(s,x)\in[0,T]\times E$.
\end{proposition}
We recall that if two functions $f, g$ differ 
only on a zero potential set then they represent the same element of $L^0_X$.
We recall our notion of \textbf{extended generator}.
\begin{definition}\label{domainextended}
 We first define the \textbf{extended domain} $\mathcal{D}(\mathfrak{a})$ as the set functions $\phi\in\mathcal{B}([0,T]\times E,\mathbbm{R})$ for which there exists 
$\psi\in\mathcal{B}([0,T]\times E,\mathbbm{R})$ such that under any $\mathbbm{P}^{s,x}$ the process
\begin{equation} \label{E45}
     \mathds{1}_{[s,T]}\left(\phi(\cdot,X_{\cdot}) - \phi(s,x) - \int_s^{\cdot}\psi(r,X_r)dr \right), 
\end{equation}
(which is not necessarily c\`adl`ag) has a c\`adl\`ag modification in $\mathcal{H}^2_{0}$. 
\end{definition}

Proposition 4.16 in \cite{paper1preprint}
states the following.
\begin{proposition}\label{uniquenesspsi}
	Let $\phi \in {\mathcal B}([0,T] \times E, {\mathbbm R}).$
There is at most one (up to zero potential sets)  $\psi 
\in {\mathcal B}([0,T] \times E, {\mathbbm R})$ such that
under any $\mathbbm{P}^{s,x}$, the process defined in \eqref{E45} 
has a modification which belongs to ${\mathcal M}_{loc}$.
	
	If moreover $\phi\in\mathcal{D}(a)$, then $a(\phi)=\psi$ up to zero potential sets. In this case, according to Notation \ref{Mphi},
 for every $(s,x)\in[0,T]\times E$, 
  $M[\phi]^{s,x}$ is the $\mathbbm{P}^{s,x}$ c\`adl\`ag modification
 in $\mathcal{H}^2_{0}$ of
	 $\mathds{1}_{[s,T]}\left(\phi(\cdot,X_{\cdot}) - \phi(s,x) - \int_s^{\cdot}\psi(r,X_r)dr \right)$.
\end{proposition}

\begin{definition}\label{extended}
Let $\phi \in \mathcal{D}(\mathfrak{a})$ as in Definition
 \ref{domainextended}.
 We denote again  by $M[\phi]^{s,x}$, the unique
 c\`adl\`ag version
 of the process \eqref{E45} in $\mathcal{H}^2_{0}$.
Taking Proposition \ref{uniquenessupto} into account, this will not
 generate  any ambiguity with respect
to Notation \ref{Mphi}.   
Proposition \ref{uniquenessupto}, also permits to define without ambiguity the operator  
\begin{equation*}
\mathfrak{a}:
\begin{array}{rcl}
\mathcal{D}(\mathfrak{a})&\longrightarrow& L^0_X\\
\phi &\longmapsto & \psi.
\end{array}
\end{equation*}
$\mathfrak{a}$ will be called the \textbf{extended generator}.
\end{definition}

We also extend the carr\'e du champs operator $\Gamma(\cdot,\cdot)$
 to $\mathcal{D}(\mathfrak{a})\times\mathcal{D}(\mathfrak{a})$. 
Proposition 4.18 in 
\cite{paper1preprint}
states the following.
\begin{proposition} \label{P321}
Let $\phi$ and $\psi$ be in $\mathcal{D}(\mathfrak{a})$, there exists a (unique up to zero-potential sets) function in $\mathcal{B}([0,T]\times E,\mathbbm{R})$ which we will denote $\mathfrak{G}(\phi,\psi)$ such that under any $\mathbbm{P}^{s,x}$, 
$\langle M[\phi]^{s,x},M[\psi]^{s,x}\rangle=\int_s^{\cdot}\mathfrak{G}(\phi,\psi)(r,X_r)dr$ on $[s,T]$, up to indistinguishability.
If moreover $\phi$ and $\psi$ belong to $\mathcal{D}(a)$, then $\Gamma(\phi,\psi)=\mathfrak{G}(\phi,\psi)$ up to zero potential sets.
\end{proposition} 
\begin{notation}
	$\mathfrak{G}(\phi,\phi)$ will be denoted  $\mathfrak{G}(\phi)$ and if that function takes positive values, $\mathfrak{G}(\phi)^{\frac{1}{2}}$ will denotes its point-wise square root.
\end{notation}

\begin{definition}\label{extendedgamma}
The bilinear operator $\mathfrak{G}:\mathcal{D}(\mathfrak{a})\times\mathcal{D}(\mathfrak{a})\longmapsto L^0_X$ will be called  the \textbf{extended carr\'e du champs operator}.
\end{definition}
According to Definition \ref{domainextended}, we do not have
necessarily  $\mathcal{D}(a)\subset\mathcal{D}(\mathfrak{a})$,
however we have the following.

\begin{corollary}\label{RExtendedClassical} 
If $\phi\in\mathcal{D}(a)$ and $\Gamma(\phi)\in\mathcal{L}^1_X$,
 then $\phi\in\mathcal{D}(\mathfrak{a})$ and $(a(\phi),\Gamma(\phi))=(\mathfrak{a}(\phi),\mathfrak{G}(\phi))$ up to zero potential sets.
\end{corollary}

We also recall Lemma 5.12 of 
\cite{paper1preprint}.
\begin{lemma}\label{ModifImpliesdV}
	Let $(s,x)\in[0,T]\times E$ be fixed and let $\phi,\psi$ be two measurable processes. If $\phi$ and $\psi$ are $\mathbbm{P}^{s,x}$-modifications of each other, then they are equal $dt\otimes d\mathbbm{P}^{s,x}$ a.e.
\end{lemma}

We now keep in mind the Pseudo-Partial Differential Equation (in short 
Pseudo-PDE), with final condition,
that we have introduced in 
\cite{paper1preprint}.
\\
Let us consider the following data.
\begin{enumerate}
	\item A measurable final condition
	$g\in\mathcal{B}(E,\mathbbm{R})$;
	\item a measurable nonlinear function
	$f\in\mathcal{B}([0,T]\times E\times\mathbbm{R}\times\mathbbm{R},\mathbbm{R})$.
\end{enumerate}
The equation is
\begin{equation}\label{PDE}
\left\{
\begin{array}{rccc}
a(u) + f\left(\cdot,\cdot,u,\Gamma(u)^{\frac{1}{2}}\right)&=&0& \text{ on } [0,T]\times E   \\
u(T,\cdot)&=&g.& 
\end{array}\right.
\end{equation}
\begin{notation}
	Equation \eqref{PDE} will be denoted $Pseudo-PDE(f,g)$.
\end{notation}
\begin{definition}\label{MarkovPDE}
	We will say that $u$ is a \textbf{classical solution} of 
	$Pseudo-PDE(f,g)$
	if it belongs to $\mathcal{D}(a)$ and verifies \eqref{PDE}.
\end{definition}

\begin{definition} \label{D417}
	A function $u: [0,T] \times E \rightarrow {\mathbbm R}$
	will be said 
	to be a {\bf martingale solution} of $Pseudo-PDE(f,g)$ if 
	$u\in\mathcal{D}(\mathfrak{a})$  and
	\begin{equation}\label{PDEextended}
	\left\{\begin{array}{rcl}
	\mathfrak{a}(u)&=& -f(\cdot,\cdot,u,\mathfrak{G}(u)^{\frac{1}{2}})\\
	u(T,\cdot)&=&g.
	\end{array}\right.
	\end{equation}
\end{definition}
We now fix couple of functions  
$f\in\mathcal{B}([0,T]\times E\times\mathbbm{R}\times\mathbbm{R},\mathbbm{R})$ and $g\in\mathcal{B}(E,\mathbbm{R})$.
Until the end of these preliminaries, we will  assume the following.
\begin{hypothesis}\label{Hpq}
\begin{enumerate}\
		\item $\forall (s,x)\in[0,T]\times E,\quad g(X_T)\in L^2(\mathbbm{P}^{s,x})$;
		\item $f(\cdot,\cdot,0,0)\in\mathcal{L}^2_X$;
		\item There exists $K^Y>0, K^Z>0$ such that for all $t,x,y,y',z,z'$, \\  $|f(t,x,y,z)-f(t,x,y',z')|\leq K^Y|y-y'|+K^Z|z-z'|$.
	\end{enumerate}
\end{hypothesis}

\begin{remark} \label{R234}
  If $f(\cdot,\cdot,0,0)$ and $g$ are bounded then properties
  1. and 2. above are satisfied.
\end{remark}

We conclude these preliminaries by stating the theorem of existence and uniqueness of a martingale solution for $Pseudo-PDE(f,g)$. It was the
object of Theorem 5.21 of \cite{paper1preprint}. 

\begin{theorem} \label{RMartExistenceUniqueness}
  Let $(\mathbbm{P}^{s,x})_{(s,x)\in[0,T]\times E}$ be a Markov canonical
  class associated to a transition function measurable 
in time (see Definitions \ref{defMarkov} and \ref{DefFoncTrans}) which
fulfills Hypothesis \ref{MPwellposed},
i.e. it is a solution of a well-posed Martingale Problem associated with
some $(\mathcal{D}(a),a)$.
Moreover assume that  Hypothesis \ref{Hpq} holds.

Then $Pseudo-PDE(f,g)$ has a unique martingale solution.
\end{theorem}

We also had shown (see Proposition 5.19 in 
  \cite{paper1preprint}) that the unique martingale solution is the only possible classical solution if there is one, as stated below.
\begin{proposition}\label{CoroClassic}
Under the conditions of previous Theorem \ref{RMartExistenceUniqueness}, 
a classical solution $u$ of $Pseudo-PDE(f,g)$ such that 
$\Gamma(u)\in\mathcal{L}^1_X$,  is also a martingale solution.	

Conversely, if $u$ is a martingale solution of $Pseudo-PDE(f,g)$  
belonging to $\mathcal{D}(a)$, then $u$ is a classical solution of $Pseudo-PDE(f,g)$ up to a zero-potential set, meaning that the first equality of \eqref{PDE} holds up to a set of zero potential. 
\end{proposition}

\section{Decoupled mild solutions of Pseudo-PDEs}\label{S2}

All along this section we will consider a  Markov canonical class $(\mathbbm{P}^{s,x})_{(s,x)\in[0,T]\times E}$ associated to a transition function $p$ measurable in time (see Definitions \ref{defMarkov}, \ref{DefFoncTrans}) verifying Hypothesis   \ref{MPwellposed} for a certain $(\mathcal{D}(a),a)$.   We are also given a couple of functions   $f\in\mathcal{B}([0,T]\times E\times\mathbbm{R}\times\mathbbm{R},\mathbbm{R})$ and $g\in\mathcal{B}(E,\mathbbm{R})$ satisfying Hypothesis \ref{Hpq}.

\subsection{Definition}\label{mild}
 
As mentioned in the introduction, in this section we introduce
 a notion of  solution of 
our $Pseudo-PDE(f,g)$ that we will denominate {\it decoupled mild}, which 
 is a generalization of the mild solution concept for partial differential equation. 
We will show that such solution exists and is unique.  
Indeed, that function  will be the one appearing in Theorem \ref{Defuv}.


In what follows, we will be interested in functions $(f,g)$ which satisfy weaker
 conditions than Hypothesis \ref{Hpq}
  namely the following ones.
\begin{hypothesis}\label{Hpqeq}
There exists $C>0$ such that the following holds.
	\begin{enumerate}
		\item $\forall (s,x)\in[0,T]\times E,\quad g(X_T)\in L^2(\mathbbm{P}^{s,x})$;
		\item $f(\cdot,\cdot,0,0)\in\mathcal{L}^2_X$;
		\item $\forall (t,x,y,z):\quad  |f(t,x,y,z)|\leq |f(t,x,0,0)| + C(|y|+|z|)$.
	\end{enumerate}
\end{hypothesis}
\begin{notation}
Let $s,t$ in $[0,T]$ with $s\leq t$, $x\in E$ and $\phi\in \mathcal{B}(E,\mathbbm{R})$, if the expectation $\mathbbm{E}^{s,x}[|\phi(X_t)|]$ is finite, then $P_{s,t}[\phi](x)$ will denote $\mathbbm{E}^{s,x}[\phi(X_t)]$ .
\end{notation}

We recall two important measurability properties.
\begin{remark} Let $\phi\in\mathcal{B}(E,\mathbbm{R})$.
\begin{itemize}
\item Suppose  that for any $(s,x,t)$, 
$\mathbbm{E}^{s,x}[|\phi(X_t)|]<\infty$ then by Proposition \ref{measurableint}, $(s,x,t)\longmapsto P_{s,t}[\phi](x)$ is Borel.
\item Suppose that for every $(s,x)$, $\mathbbm{E}^{s,x}[\int_s^{T}|\phi(X_r)|dr]<\infty$. Then by Lemma \ref{LemmaBorel},
$(s,x)\longmapsto\int_s^TP_{s,r}[\phi](x)dr$ is Borel.
\end{itemize}
\end{remark}

In our  general setup, considering some operator $a$, the equation 
\begin{equation}
a(u) + f\left(\cdot,\cdot,u,\Gamma(u)^{\frac{1}{2}}\right)=0,
\end{equation}
 can be naturally  decoupled into 
\begin{equation}
\left\{
\begin{array}{rcl}
a(u) &=& - f(\cdot,\cdot,u,v)\\
\Gamma(u) &=&  v^2.
\end{array}\right.
\end{equation}
Since $\Gamma(u)=a(u^2)-2ua(u)$, this system of equation will be rewritten as
\begin{equation}
\left\{
\begin{array}{rcl}
a(u) &=& - f(\cdot,\cdot,u,v)\\
a(u^2) &=& v^2 - 2uf(\cdot,\cdot,u,v).
\end{array}\right.
\end{equation}
On the other hand our Markov process $X$ is time non-homogeneous, which  leads us to the definition of a decoupled mild solution.

\begin{definition}\label{mildsoluv}
Let $(f,g)$ be a couple verifying Hypothesis \ref{Hpqeq}. 
\\
Let $u,v\in\mathcal{B}([0,T]\times E,\mathbbm{R})$ be two Borel functions with $v\geq 0$.
\begin{enumerate}
\item The couple $(u,v)$ will be called {\bf solution
of the identification problem determined by $(f,g)$}
or simply {\bf  solution of } 
$IP(f,g)$ if  $u$ and $v$ belong to $\mathcal{L}^2_X$ and if for every $(s,x)\in[0,T]\times E$,
\begin{equation}\label{MildEq}
\left\{
    \begin{array}{rcl}
    u(s,x)&=&P_{s,T}[g](x)+\int_s^TP_{s,r}\left[f\left(r,\cdot,u(r,\cdot),v(r,\cdot)\right)\right](x)dr\\
    u^2(s,x) &=&P_{s,T}[g^2](x) -\int_s^TP_{s,r}\left[v^2(r,\cdot)-2uf\left(r,\cdot,u(r,\cdot),v(r,\cdot)\right)\right](x)dr.
    \end{array}\right.
\end{equation}
\item The function $u$ will be called  {\bf decoupled mild solution}
of $Pseudo-PDE(f,g)$ if there is a function $v$
such that the couple $(u,v)$ is a solution 
of $IP(f,g)$.
\end{enumerate}
\end{definition}

\begin{lemma}\label{LemmaMild}
	Let $u,v\in\mathcal{L}^2_X$, and let $f$ be a Borel function satisfying Hypothesis \ref{Hpqeq}, then $f\left(\cdot,\cdot,u,v\right)$ belongs to $\mathcal{L}_X^2$ and $uf\left(\cdot,\cdot,u,v\right)$ to $\mathcal{L}^1_X$. 
\end{lemma}

\begin{proof}
	Thanks to the growth condition on $f$ in Hypothesis \ref{Hpqeq}, there exists a constant $C>0$ such that for any $(s,x)\in[0,T]\times E$,
	\begin{equation}\label{EqUniqueness}
	\begin{array}{rcl}
	&& \mathbbm{E}^{s,x}\left[\int_t^T f^2(r,X_r,u(r,X_r),v(r,X_r))dr\right]\\
	&\leq& C\mathbbm{E}^{s,x}\left[\int_t^T(f^2(r,X_r,0,0) +u^2(r,X_r)+v^2(r,X_r))dr\right]
	< \infty,
	\end{array}
	\end{equation}
	since we have assumed that $u^2,v^2$ belong to $\mathcal{L}^1_X$, and thanks to Hypothesis \ref{Hpqeq}. This means that $f^2\left(\cdot,\cdot,u,v\right)$  belongs to $\mathcal{L}^1_X$. Since $2\left|uf\left(\cdot,\cdot,u,v\right)\right|\leq u^2+f^2\left(\cdot,\cdot,u,v\right)$  then  $uf\left(\cdot,\cdot,u,v\right)$  also belongs to $\mathcal{L}^1_X$.
\end{proof}
\begin{remark} Consequently,
	under the assumptions of Lemma \ref{LemmaMild} 
all the terms in \eqref{MildEq} make sense.
\end{remark}

\subsection{Existence and uniqueness of a solution}\label{mart-mild}

\begin{proposition}\label{MartingaleImpliesMild} 
Assume that $(f,g)$ verifies Hypothesis \ref{Hpqeq} and let $u\in\mathcal{L}^2_X$ be a martingale solution of $Pseudo-PDE(f,g)$. 
Then $(u,\mathfrak{G}(u))$ is a
 solution of $IP(f,g)$ and in particular, $u$ is a decoupled mild solution of $Pseudo-PDE(f,g)$.
\end{proposition}

\begin{proof}
Let $u\in\mathcal{L}^2_X$ be a martingale solution of $Pseudo-PDE(f,g)$. We emphasize that, taking Definition \ref{domainextended} and Proposition \ref{P321} into account, $\mathfrak{G}(u)$ belongs to $\mathcal{L}^1_X$, 
or equivalently that $\mathfrak{G}( u)^{\frac{1}{2}}$ belongs to $\mathcal{L}^2_X$.  By Lemma \ref{LemmaMild}, it follows that
 $f\left(\cdot,\cdot,u,\mathfrak{G}( u)^{\frac{1}{2}}\right)\in\mathcal{L}^2_X$ and $uf\left(\cdot,\cdot,u,\mathfrak{G}( u)^{\frac{1}{2}}\right)\in\mathcal{L}^1_X$.
\\
We fix some $(s,x)\in[0,T]\times E$ and the corresponding probability $\mathbbm{P}^{s,x}$. We are going to show that 
\begin{equation} \label{MildEqAux}
\left\{
\begin{array}{rcl}
u(s,x)&=&P_{s,T}[g](x)+\int_s^TP_{s,r}\left[f\left(r,\cdot,u(r,\cdot),\mathfrak{G}( u)^{\frac{1}{2}}(r,\cdot)\right)\right](x)dr\\
u^2(s,x) &=&P_{s,T}[g^2](x) -\int_s^TP_{s,r}\left[\mathfrak{G}(u)(r,\cdot)-2uf\left(r,\cdot,u(r,\cdot),\mathfrak{G}( u)^{\frac{1}{2}}(r,\cdot)\right)\right](x)dr.
\end{array}\right.
\end{equation}
Combining Definitions \ref{domainextended}, \ref{extended}, \ref{D417}, we know that on $[s,T]$, the process $u(\cdot,X_{\cdot})$ has a c\`adl\`ag 
modification which we denote $U^{s,x}$ which is a special semimartingale with decomposition
\begin{equation}\label{decompoU}
	U^{s,x}=u(s,x)-\int_s^{\cdot}f\left(\cdot,\cdot,u,\mathfrak{G}( u)^{\frac{1}{2}}\right)(r,X_r)dr +M[u]^{s,x},
\end{equation}
where $M[u]^{s,x}\in\mathcal{H}^2_0$.
Definition \ref{D417} also states that $u(T,\cdot)=g$, implying that
\begin{equation}
u(s,x)=g(X_T)+\int_s^Tf\left(\cdot,\cdot,u,\mathfrak{G}( u)^{\frac{1}{2}}\right)(r,X_r)dr -M[u]^{s,x}_T\text{ a.s.}
\end{equation}
Taking the expectation, by Fubini's theorem we get
\begin{equation}
\begin{array}{rcl}
u(s,x) &=&\mathbbm{E}^{s,x}\left[g(X_T) +\int_s^T f\left(\cdot,\cdot,u,\mathfrak{G}( u)^{\frac{1}{2}}\right)(r,X_r)dr\right]\\
&=&P_{s,T}[g](x)+\int_s^TP_{s,r}\left[f\left(r,\cdot,u(r,\cdot),\mathfrak{G}( u)^{\frac{1}{2}}(r,\cdot)\right)\right](x)dr.
\end{array}
\end{equation}
By integration by parts,  we obtain
\begin{equation}
	d(U^{s,x})^2_t=-2U^{s,x}_tf\left(\cdot,\cdot,u,\mathfrak{G}( u)^{\frac{1}{2}}\right)(t,X_t)dt+2U^{s,x}_{t^-}dM[u]^{s,x}_t +d[M[u]^{s,x}]_t,
\end{equation} 
so integrating from $s$ to $T$, we get
\begin{equation}
\begin{array}{rcl}\label{Eq322}
&&u^2(s,x)\\
&=& g^2(X_T)+2\int_s^TU_r^{s,x}f\left(\cdot,\cdot,u,\mathfrak{G}( u)^{\frac{1}{2}}\right)(r,X_r)dr -2\int_s^TU^{s,x}_{r^-}dM[u]^{s,x}_r -[M[u]^{s,x}]_T\\
&=&g^2(X_T)+2\int_s^Tuf\left(\cdot,\cdot,u,\mathfrak{G}( u)^{\frac{1}{2}}\right)(r,X_r)dr -2\int_s^TU^{s,x}_{r^-}dM[u]^{s,x}_r -[M[u]^{s,x}]_T,
\end{array}
\end{equation}
where the latter line is a consequence of Lemma \ref{ModifImpliesdV}.
The next step will consist in taking the expectation in equation \eqref{Eq322}, but before, we will check that $\int_s^{\cdot}U^{s,x}_{r^-}dM[u]^{s,x}_r$ is a martingale.
Thanks to \eqref{decompoU} and Jensen's inequality, there exists a constant $C>0$ such that 
\begin{equation}
	\underset{t\in[s,T]}{\text{sup }}(U^{s,x}_t)^2\leq C\left(\int_s^Tf^2\left(\cdot,\cdot,u,\mathfrak{G}( u)^{\frac{1}{2}}\right)(r,X_r)dr +\underset{t\in[s,T]}{\text{sup }}(M[u]^{s,x}_t)^2\right).
\end{equation}
Since $M[u]^{s,x}\in\mathcal{H}^2_0$ and $f\left(\cdot,\cdot,u,\mathfrak{G}( u)^{\frac{1}{2}}\right)\in\mathcal{L}^2_X$, 
it follows that $\underset{t\in[s,T]}{\text{sup }}(U^{s,x}_t)^2\in L^1(\mathbbm{P}^{s,x})$ and Lemma 3.15 in
\cite{paper1preprint}
states that $\int_s^{\cdot}U^{s,x}_{r^-}dM[u]^{s,x}_r$ is a $\mathbbm{P}^{s,x}$-martingale.
Taking the expectation in \eqref{Eq322}, we now obtain
\begin{equation}
\begin{array}{rcl}
u^2(s,x) &=& \mathbbm{E}^{s,x}\left[g^2(X_T) +\int_s^T 2uf\left(\cdot,\cdot,u,\mathfrak{G}( u)^{\frac{1}{2}}\right)(r,X_r)dr-\left[M[u]^{s,x}\right]_T\right]\\
&=& \mathbbm{E}^{s,x}\left[g^2(X_T) +\int_s^T 2uf\left(\cdot,\cdot,u,\mathfrak{G}( u)^{\frac{1}{2}}\right)(r,X_r)dr-\langle M[u]^{s,x}\rangle_T\right]\\
&=&\mathbbm{E}^{s,x}\left[g^2(X_T)\right]- \mathbbm{E}^{s,x}\left[\int_s^T \left(\mathfrak{G}(u)-2uf\left(\cdot,\cdot,u,\mathfrak{G}( u)^{\frac{1}{2}}\right)\right)(r,X_r)dr\right]\\
&=& P_{s,T}[g^2](x) -\int_s^TP_{s,r}\left[\mathfrak{G}(u)(r,\cdot)-2u(r,\cdot)f\left(r,\cdot,u(r,\cdot),\mathfrak{G}( u)^{\frac{1}{2}}(r,\cdot)\right)\right](x)dr,
\end{array}
\end{equation}
where the third equality derives from Proposition \ref{P321} and the fourth from Fubini's theorem.
This concludes the proof.
\end{proof}

We now show the converse result of Proposition \ref{MartingaleImpliesMild}.

\begin{proposition}\label{MildImpliesMartingale}  
Assume that $(f,g)$ verifies Hypothesis \ref{Hpqeq}. 
Every decoupled mild solution of $Pseudo-PDE(f,g)$ is a also a martingale solution.   Moreover, if $(u,v)$ solves $IP(f,g)$, then $v^2=\mathfrak{G}(u)$ (up to zero potential sets).
\end{proposition}

\begin{proof} 
Let $u$ and $v\geq 0$  be a couple of functions in $\mathcal{L}^2_X$ verifying \eqref{MildEq}.  We first note that, the first line of \eqref{MildEq} with $s=T$, gives $u(T,\cdot)=g$. 
We fix $(s,x)\in[0,T]\times E$ and the associated probability $\mathbbm{P}^{s,x}$, and on $[s,T]$,
  we set $U_t:=u(t,X_t)$ and $N_t:=u(t,X_t)-u(s,x)+\int_s^tf(r,X_r,u(r,X_r),v(r,X_r))dr$.

Combining the first line of \eqref{MildEq} applied in $(s,x)=(t,X_t)$ and the Markov property \eqref{Markov3}, and since $f\left(\cdot,\cdot,u,v\right)$ belongs to $\mathcal{L}^2_X$ (see Lemma \ref{LemmaMild}) we get the a.s. equalities
\begin{equation}
    \begin{array}{rcl}
     U_t &=& u(t,X_t) \\
     &=& P_{t,T}[g](X_t) + \int_t^TP_{t,r}\left[f\left(r,\cdot,u(r,\cdot),v(r,\cdot)\right)\right](X_t)dr\\
     &=& \mathbbm{E}^{t,X_t}\left[g(X_T)+\int_t^T f(r,X_r,u(r,X_r),v(r,X_r))dr\right] \\
     &=& \mathbbm{E}^{s,x}\left[g(X_T)+\int_t^T f(r,X_r,u(r,X_r),v(r,X_r))dr|\mathcal{F}_t\right],
    \end{array}
\end{equation}
from which we deduce that
$N_t=\mathbbm{E}^{s,x}\left[g(X_T)+\int_s^T f(r,X_r,u(r,X_r),v(r,X_r))dr|\mathcal{F}_t\right]-u(s,x)$ a.s.
 So $N$  is a $\mathbbm{P}^{s,x}$-martingale. We can therefore consider on $[s,T]$ and under $\mathbbm{P}^{s,x}$, $N^{s,x}$ the c\`adl\`ag version of $N$, and  the special semi-martingale 
\\
$U^{s,x} := u(s,x)- \int_s^{\cdot} f(r,X_r,u(r,X_r),v(r,X_r))dr + N^{s,x}$ which is a c\`adl\`ag  version of $U$.
By Jensen's inequality  for both expectation and
 conditional expectation, we have
\begin{equation}
\begin{array}{rcl}
\mathbbm{E}^{s,x}[(N^{s,x})^2_t]&=&\mathbbm{E}^{s,x}\left[\left(\mathbbm{E}^{s,x}\left[g(X_T)+\int_s^T f(r,X_r,u(r,X_r),v(r,X_r))dr|\mathcal{F}_t\right]-u(s,x)\right)^2\right]\\
&\leq& 3u^2(s,x) + 3\mathbbm{E}^{s,x}[g^2(X_T)] + 3\mathbbm{E}^{s,x}\left[\int_s^T f^2(r,X_r,u(r,X_r),v(r,X_r))dr\right]\\
&<& \infty,
\end{array}
\end{equation}
where the  second term is finite because of Hypothesis \ref{Hpqeq}, and the same  also holds for  the third one because $f\left(\cdot,\cdot,u,v\right)$ belongs to $\mathcal{L}^2_X$, see Lemma \ref{LemmaMild}. So $N^{s,x}$ is square integrable. We have therefore shown that under any $\mathbbm{P}^{s,x}$, the process $u(\cdot,X_{\cdot})-u(s,x)+\int_s^{\cdot}f(r,X_r,u(r,X_r),v(r,X_r))dr$ has on $[s,T]$ a modification in $\mathcal{H}^2_0$.  Definitions \ref{domainextended} and \ref{extended}, justify  that $u\in\mathcal{D}(\mathfrak{a})$, 
$\mathfrak{a}(u)=-f(\cdot,\cdot,u,v)$ and that for any $(s,x)\in[0,T]\times E$, $M[u]^{s,x}=N^{s,x}$. 

To conclude that $u$ is a martingale solution of $Pseudo-PDE(f,g)$, there is left to show that $\mathfrak{G}(u)=v^2$, up to zero potential sets. By Proposition \ref{P321}, this is equivalent to show that for every $(s,x)\in[0,T]\times E$,  
$\langle N^{s,x}\rangle = \int_s^{\cdot}v^2(r,X_r)dr$, 
in the sense of $\mathbbm{P}^{s,x}$-indistinguishability. 

We fix again $(s,x)\in[0,T]\times E$ and the associated probability, and  now set 
\\
$N'_t:=u^2(t,X_t)-u^2(s,x)-\int_s^t(v^2-2uf(\cdot,\cdot,u,v))(r,X_r)dr$. Combining the second line of \eqref{MildEq} applied in $(s,x)=(t,X_t)$ and the Markov property \eqref{Markov3}, and since $v^2,uf\left(\cdot,\cdot,u,v\right)$ belong to $\mathcal{L}^1_X$ (see Lemma \ref{LemmaMild}) we get the a.s. equalities
\begin{equation}
    \begin{array}{rcl}
     u^2(t,X_t) &=& P_{t,T}[g^2](X_t) - \int_t^TP_{t,r}\left[(v^2(r,\cdot)-2u(r,\cdot)f\left(r,\cdot,u(r,\cdot),v(r,\cdot)\right))\right](X_t)dr\\
     &=& \mathbbm{E}^{t,X_t}\left[g^2(X_T)-\int_t^T(v^2-2uf(\cdot,\cdot,u,v))(r,X_r)dr\right] \\
     &=& \mathbbm{E}^{s,x}\left[g^2(X_T)-\int_t^T(v^2- 2uf(\cdot,\cdot,u,v))(r,X_r)dr|\mathcal{F}_t\right],
    \end{array}
\end{equation}
from which we deduce that for any $t\in[s,T]$,
$$ N'_t=\mathbbm{E}^{s,x}\left[g^2(X_T)-\int_s^T (v^2- uf(\cdot,\cdot,u,v))(r,X_r)dr|\mathcal{F}_t\right]-u^2(s,x) \ {\rm a.s.}$$
 So $N'$  is a $\mathbbm{P}^{s,x}$-martingale. We can therefore consider on $[s,T]$ and under $\mathbbm{P}^{s,x}$, $N'^{s,x}$ the c\`adl\`ag version of $N'$. 

 The process $u^2(s,x)+ \int_s^{\cdot}(v^2-uf(\cdot,\cdot,u,v))(r,X_r)dr + N'^{s,x}$ is therefore a c\`adl\`ag special semi-martingale which is a $\mathbbm{P}^{s,x}$-version of $u^2(\cdot,X)$ on $[s,T]$. But we also had shown that 
 $U^{s,x}=u(s,x)- \int_s^{\cdot} f(r,X_r,u(r,X_r),v(r,X_r))dr + N^{s,x},$ is a version of $u(\cdot,X),$ which by integration by parts implies that 
 $$ u^2(s,x)-2\int_s^{\cdot}U^{s,x}_rf(\cdot,\cdot,u,v)(r,X_r)dr+2\int_s^{\cdot}U^{s,x}_{r^-}dN^{s,x}_r+[N^{s,x}],$$
 is another c\`adl\`ag  semi-martingale which is a $\mathbbm{P}^{s,x}$-version of $u^2(\cdot,X)$ on $[s,T]$. 
\\
$\int_s^{\cdot}(v^2-2uf(\cdot,\cdot,u,v))(r,X_r)dr + N'^{s,x}$ is therefore indistinguishable from  
\\
$-2\int_s^{\cdot}U^{s,x}_rf(\cdot,\cdot,u,v)(r,X_r)dr+2\int_s^{\cdot}U^{s,x}_{r^-}dN^{s,x}_r+[N^{s,x}]$, which can be written $\langle N^{s,x}\rangle-2\int_s^{\cdot}U^{s,x}_rf(\cdot,\cdot,u,v)(r,X_r)dr +2\int_s^{\cdot}U^{s,x}_{r^-}dN^{s,x}_r+([N^{s,x}]-\langle N^{s,x}\rangle)$ where $\langle N^{s,x}\rangle-2\int_s^{\cdot}U^{s,x}_rf(\cdot,\cdot,u,v)(r,X_r)dr$ is predictable with bounded variation and $2\int_s^{\cdot}U^{s,x}_{r^-}dN^{s,x}_r+([N^{s,x}]-\langle N^{s,x}\rangle)$ is a local martingale. By uniqueness of the decomposition of a special semi-martingale, we have 
$$\int_s^{\cdot}(v^2-2uf(\cdot,\cdot,u,v))(r,X_r)dr=\langle N^{s,x}\rangle-2\int_s^{\cdot}U^{s,x}_rf(\cdot,\cdot,u,v)(r,X_r)dr$$
and by Lemma \ref{ModifImpliesdV}, 
$$\int_s^{\cdot}(v^2-2uf(\cdot,\cdot,u,v))(r,X_r)dr=\langle N^{s,x}\rangle-2\int_s^{\cdot}uf(\cdot,\cdot,u,v)(r,X_r)dr,
$$
which finally yields $\langle N^{s,x}\rangle=\int_s^{\cdot}v^2(r,X_r)dr$ as desired.
\end{proof}

We recall that $(\mathbbm{P}^{s,x})_{(s,x)\in[0,T]\times E}$ is
 a Markov canonical class associated to a transition function measurable 
in time (see Definitions \ref{defMarkov} and \ref{DefFoncTrans}) which
fulfills Hypothesis \ref{MPwellposed},
i.e. it is a solution of a well-posed Martingale Problem associated with $(\mathcal{D}(a),a)$.
\begin{theorem}\label{MainTheorem}
 Let $(f,g)$ be a couple verifying Hypothesis
\ref{Hpq}.
Then $Pseudo-PDE(f,g)$ has a unique decoupled mild solution.
\end{theorem}
\begin{proof}
This derives from Theorem \ref{RMartExistenceUniqueness} and Propositions \ref{MartingaleImpliesMild}, \ref{MildImpliesMartingale}.
\end{proof}

\begin{corollary}\label{CoroClassicMild}
Assume that $(f,g)$ verifies  Hypothesis \ref{Hpq}.
A classical solution $u$ of $Pseudo-PDE(f,g)$ such that 
$\Gamma(u)\in\mathcal{L}^1_X$,  is also a decoupled mild solution.	

Conversely, if $u$ is a decoupled mild solution of $Pseudo-PDE(f,g)$  
belonging to $\mathcal{D}(a)$, then $u$ is a classical solution of $Pseudo-PDE(f,g)$ up to a zero-potential set, meaning that the first equality of \eqref{PDE} holds up to a set of zero potential. 
\end{corollary}

\begin{proof}
The statement holds by Proposition \ref{MildImpliesMartingale} and Proposition \ref{CoroClassic}.
\end{proof}

\subsection{Representation of the solution via FBSDEs with no driving martingale}\label{FBSDEs}

In the companion paper \cite{paper1preprint}, the following family of FBSDEs with no driving martingale indexed by $(s,x)\in[0,T]\times E$ was introduced.
\begin{definition}
	Let  $(s,x)\in[0,T]\times E$  and the associated stochastic basis $\left(\Omega,\mathcal{F}^{s,x},(\mathcal{F}^{s,x}_t)_{t\in[0,T]},\mathbbm{P}^{s,x}\right)$ be fixed. A couple 
	$(Y^{s,x},M^{s,x})\in\mathcal{L}^2(dt\otimes d\mathbbm{P}^{s,x})\times\mathcal{H}^2_0$,
        will be said to solve $FBSDE^{s,x}(f,g)$ if it verifies on $[0,T]$,
 in the sense of indistinguishability
	\begin{equation}\label{BSDE}
	Y^{s,x} = g(X_T) + \int_{\cdot}^T f\left(r,X_r,Y^{s,x}_r,\sqrt{\frac{d\langle M^{s,x}\rangle_r}{dr}}\right)dr  -(M^{s,x}_T - M^{s,x}_{\cdot}).
	\end{equation}
If \eqref{BSDE} is only satisfied on a smaller interval $[t_0,T]$, with $0<t_0<T$, we say that $(Y^{s,x},M^{s,x})$ solves $FBSDE^{s,x}(f,g)$ on $[t_0,T]$.
\end{definition}

The following result follows from Theorem 3.22 in \cite{paper1preprint}.
\begin{theorem}\label{SolBSDE}
Assume that $(f,g)$ verifies  Hypothesis \ref{Hpq}.
 Then for any 
$(s,x)\in[0,T]\times E$,  $FBSDE^{s,x}(f,g)$ has a unique solution.
\end{theorem}

In the following theorem, we summarize the links between the $FBSDE^{s,x}(f,g)$ and the notion of martingale solution of $Pseudo-PDE(f,g)$.
 These are shown in Theorem 5.14, Remark 5.15, Theorem 5.20 and Theorem 5.21 of \cite{paper1preprint}.
 
\begin{theorem}\label{Defuv}
Assume that $(f,g)$ verifies  \ref{Hpq}
 and let $(Y^{s,x}, M^{s,x})$ denote the (unique) solution of $FBSDE^{s,x}(f,g)$ for fixed $(s,x)$.
Let $u$ be the unique martingale solution 
of $Pseudo-PDE(f,g)$.

For every 
	$(s,x)\in[0,T]\times E$, on the interval $[s,T]$, we have the following.
	\begin{itemize}
		\item $Y^{s,x}$ and $u(\cdot,X_{\cdot})$ are  $\mathbbm{P}^{s,x}$-modifications, and equal $dt\otimes d\mathbbm{P}^{s,x}$ a.e.;
		\item $M^{s,x}$ and $M[u]^{s,x}$ are $\mathbbm{P}^{s,x}$-indistinguishable.
	\end{itemize}
Moreover $u$ belongs to $\mathcal{L}^2_X$ and
 for any $(s,x)\in[0,T]\times E$, we have
 $\frac{d\langle M^{s,x}\rangle_t}{dt}=\mathfrak{G}(u)(t,X_{t})$ $dt\otimes d\mathbbm{P}^{s,x}$ a.e.
\end{theorem}
\begin{remark} The martingale solution $u$ of $Pseudo-PDE$ exists
and is unique
by Theorem \ref{RMartExistenceUniqueness}.
\end{remark}
We can therefore represent the unique decoupled mild solution of $Pseudo-PDE(f,g)$ via the stochastic equations $FBSDE^{s,x}(f,g)$ as follows.
\begin{theorem}\label{Representation}
Assume that $(f,g)$ verifies see Hypothesis
\ref{Hpq}  and let $(Y^{s,x}, M^{s,x})$ denote the (unique) solution of $FBSDE^{s,x}(f,g)$ for fixed $(s,x)$.
	
Then for any $(s,x)\in[0,T]\times E$, the random variable $Y^{s,x}_s$ is $\mathbbm{P}^{s,x}$ a.s. equal to a constant (which we still denote $Y^{s,x}_s$), and the function 
\begin{equation}
u:(s,x)\longmapsto Y^{s,x}_s
\end{equation}
is the unique decoupled mild solution of $Pseudo-PDE(f,g)$.
\end{theorem}
\begin{proof}
	By Theorem \ref{Defuv}, there exists a Borel function $u$ such that for every $(s,x)\in[0,T]\times E$, $Y^{s,x}_s=u(s,X_s)=u(s,x)$ $\mathbbm{P}^{s,x}$ a.s. and $u$ is the unique martingale solution of $Pseudo-PDE(f,g)$. By Proposition \ref{MartingaleImpliesMild}, it is also its unique decoupled mild solution.
\end{proof}

\begin{remark} \label{R317}
The function $v$ such that $(u,v)$ is the unique solution of the identification problem $IP(f,g)$ also has a stochastic representation since it verifies for every $(s,x)\in[0,T]\times E$, on the interval $[s,T]$,
\\
$\frac{d\langle M^{s,x}\rangle_t}{dt}=v^2(t,X_{t})$ $dt\otimes d\mathbbm{P}^{s,x}$ a.e. where $M^{s,x}$ is the martingale part of the solution of $FBSDE^{s,x}$.
\end{remark}

Conversely, under the weaker condition Hypothesis \ref{Hpqeq} if one knows the solution of $IP(f,g)$, one can (for every $(s,x)$) produce a version of
 a solution of $FBSDE^{s,x}(f,g)$ as follows. This is only possible with the notion of decoupled mild solution:
even in the case of Brownian BSDEs the knowledge of the viscosity solution
of the related PDE would (in general) not be sufficient to reconstruct the 
family of solutions of the BSDEs.

\begin{proposition}\label{MildImpliesBSDE}
Assume that $(f,g)$ verifies Hypothesis \ref{Hpqeq}.
Suppose the existence of a  solution $(u,v)$ to $IP(f,g)$, and let $(s,x)\in[0,T]\times E$ be fixed.
Then 
\begin{equation}
\left(u(\cdot,X),\quad u(\cdot,X)-u(s,x)+\int_s^{\cdot}f(\cdot,\cdot,u,v)(r,X_r)dr\right)
\end{equation}
admits on $[s,T]$ a $\mathbbm{P}^{s,x}$-version $(Y^{s,x},M^{s,x})$ 
which solves $FBSDE^{s,x}$ on $[s,T]$.
\end{proposition}
\begin{proof}
By Proposition \ref{MildImpliesMartingale}, $u$ is a martingale solution of $Pseudo-PDE(f,g)$ and $v^2=\mathfrak{G}(u)$. We now fix $(s,x)\in[0,T]\times E$. Combining Definitions \ref{extended}, \ref{extendedgamma} and \ref{D417}, we know that $u(T,\cdot)=g$ and that on $[s,T]$, $u(\cdot,X)$ has a $\mathbbm{P}^{s,x}$-version $U^{s,x}$ with decomposition  
$U^{s,x}=u(s,x)-\int_s^{\cdot}f(\cdot,\cdot,u,v)(r,X_r)dr +M[u]^{s,x}$, where $M[u]^{s,x}$ is an element of $\mathcal{H}^2_0$ of angular bracket $\int_s^{\cdot}v^2(r,X_r)dr$ and is a version of $u(\cdot,X)-u(s,x)+\int_s^{\cdot}f(\cdot,\cdot,u,v)(r,X_r)dr$. By Lemma \ref{ModifImpliesdV}, taking into account
$u(T,\cdot)=g$, the couple $(U^{s,x},M[u]^{s,x})$ verifies on $[s,T]$, in the sense of indistinguishability
\begin{equation}
U^{s,x} = g(X_T)+\int_{\cdot}^Tf\left(r,X_r,U^{s,x}_r,\sqrt{\frac{d\langle M[u]^{s,x}\rangle_r}{dr}}\right)dr -(M[u]^{s,x}_T-M[u]^{s,x}_{\cdot})
\end{equation}
with $M[u]^{s,x}\in\mathcal{H}^2_0$ verifying $M[u]^{s,x}_s=0$ (see Definition \ref{extended}) and $U^{s,x}_s$  is deterministic so in particular 
is a square integrable r.v.
 Following a slight adaptation of 
the proof of Lemma 3.25 in  \cite{paper1preprint}
(see Remark \ref{RAdaptation} below), this implies that $U^{s,x}\in\mathcal{L}^2(dt\otimes d\mathbbm{P}^{s,x})$ and therefore that 
that $(U^{s,x},M[u]^{s,x})$ is a solution of $FBSDE^{s,x}(f,g)$ 
  on $[s,T]$.
\end{proof}
\begin{remark} \label{RAdaptation}
Indeed Lemma 3.25 in \cite{paper1preprint}, taking into account Notation 5.5
  ibidem, can be applied rigorously only under Hypothesis \ref{Hpq} for $(f,g)$.
 However, the same proof easily
allows an extension to our framework Hypothesis \ref{Hpqeq}.

\end{remark}

\section{Examples of applications}\label{exemples}

We now develop some examples. Some of the applications that we are interested in involve operators which only act on the space variable, and we will extend them to time-dependent functions. The reader may consult 
Appendix \ref{SC}, concerning details about such extensions.
In all the items  below there will be a Markov canonical class with  transition function 
measurable in time which is solution of a well-posed Martingale Problem associated to some  $(\mathcal{D}(a),a)$ as introduced in Definition \ref{MartingaleProblem}. Therefore all the results of
 this paper will apply to all the examples below, namely Theorem \ref{RMartExistenceUniqueness}, Propositions \ref{CoroClassic}, \ref{MartingaleImpliesMild} and \ref{MildImpliesMartingale}, Theorem \ref{MainTheorem}, Corollaries \ref{CoroClassicMild} and \ref{CoroClassicMild}, Theorems \ref{SolBSDE}, \ref{Defuv} and \ref{Representation} and Proposition \ref{MildImpliesBSDE}. In particular,  Theorem \ref{MainTheorem}
states in all the cases, under suitable Lipschitz type conditions for the driver $f$,
that the corresponding Pseudo-PDE admits a unique decoupled mild solution.
In all the examples   $T \in\mathbbm{R}_+^*$ will be fixed.

\subsection{Markovian jump diffusions}\label{S4a}

In this subsection, the state space will be $E:=\mathbbm{R}^d$ for some $d\in\mathbbm{N}^*$.
We are given $\mu\in\mathcal{B}([0,T]\times \mathbbm{R}^d, \mathbbm{R}^d)$, $\alpha\in\mathcal{B}([0,T]\times \mathbbm{R}^d,S^*_+(\mathbbm{R}^d))$ 
(where $S^*_+(\mathbbm{R}^d)$ is the space of symmetric strictly positive definite matrices of size $d$) and $K$ a L\'evy kernel: this means  that for every $(t,x)\in [0,T]\times \mathbbm{R}^d$, $K(t,x,\cdot)$ is a $\sigma$-finite measure 
on $\mathbbm{R}^d\backslash\{0\}$, $\underset{t,x}{\text{sup}}\int \frac{\|y\|^2}{1+\|y\|^2}K(t,x,dy)<\infty$ and for every Borel set $A\in\mathcal{B}(\mathbbm{R}^d\backslash\{0\})$, 
$(t,x)\longmapsto \int_A \frac{\|y\|^2}{1+\|y\|^2}K(t,x,dy)$ is Borel.
 We will consider the operator $a$ defined by
\begin{equation} \label{PIDE}
\partial_t\phi + \frac{1}{2}Tr(\alpha\nabla^2\phi) + (\mu,\nabla \phi) +\int\left(\phi(\cdot,\cdot+y)-\phi(\cdot,y)-\frac{(y,\nabla \phi)}{1+\|y\|^2}\right)K(\cdot,\cdot,dy),
\end{equation}
on the domain $\mathcal{D}(a)$ which is here the linear algebra
$\mathcal{C}^{1,2}_b([0,T]\times\mathbbm{R}^d,\mathbbm{R})$ of real continuous bounded functions on $[0,T]\times \mathbbm{R}^d$ which are continuously differentiable in the first variable with bounded derivative, and twice continuously differentiable in the second variable with bounded derivatives.
\\
\\
Concerning martingale problems associated to parabolic PDE operators, one may consult \cite{stroock}. Since we want to include integral operators, we will adopt the formalism of D.W. Stroock in \cite{stroock1975diffusion}. Its Theorem 4.3 and the penultimate sentence of its proof states the following.
\begin{theorem}\label{Stroock}
Suppose that $\mu$ is bounded, that $\alpha$ is bounded continuous  
 and that for any $A\in\mathcal{B}(\mathbbm{R}^d\backslash\{0\})$, 
   $(t,x)\longmapsto \int_A \frac{y}{1+\|y\|^2}K(t,x,dy)$ is bounded continuous.
 Then, for every $(s,x)$, there exists a unique probability $\mathbbm{P}^{s,x}$ on the canonical space (see Definition \ref{canonicalspace}) such that $\phi(\cdot,X_{\cdot})-\int_s^{\cdot}a(\phi)(r,X_r)dr$ is a local martingale for any $\phi\in\mathcal{D}(a)$ and $\mathbbm{P}^{s,x}(X_s=x)=1$. 
 Moreover  $(\mathbbm{P}^{s,x})_{(s,x)\in[0,T]\times\mathbbm{R}^d}$ defines a
 Markov canonical class and its transition function is measurable in time.
\end{theorem}
The Martingale Problem associated to $(\mathcal{D}(a),a)$ 
 in the sense of Definition \ref{MartingaleProblem} is therefore well-posed and solved by $(\mathbbm{P}^{s,x})_{(s,x)\in[0,T]\times\mathbbm{R}^d}$.
In this context, $\mathcal{D}(a)$ is an algebra and for $\phi,\psi$ in $\mathcal{D}(a)$, the carr\'e du champs operator is given by
\begin{equation}
\Gamma(\phi,\psi) = \underset{i,j\leq d}{\sum}\alpha_{i,j}\partial_{x_i}\phi\partial_{x_j}\psi+\int_{\mathbbm{R}^d\backslash\{0\}}(\phi(\cdot,\cdot+y)-\phi)(\psi(\cdot,\cdot+y)-\psi)K(\cdot,\cdot,dy).\nonumber
\end{equation}

\begin{proposition}
Under the assumptions of Theorem \ref{Stroock}, and if $(f,g)$ verify Hypothesis \ref{Hpq},   $Pseudo-PDE(f,g)$ admits a unique decoupled mild solution in the sense of Definition \ref{mildsoluv}.
\end{proposition}
\begin{proof}
$\mathcal{D}(a)$ is an algebra. Moreover $(\mathbbm{P}^{s,x})_{(s,x)\in[0,T]\times\mathbbm{R}^d}$ is a Markov class which is measurable in time, and it solves the well-posed Martingale Problem associated to  $(\mathcal{D}(a),a)$. Therefore our Theorem \ref{MainTheorem} applies.
\end{proof}
We recall that if  $f$ is Lipschitz in $(y,z)$ uniformly in $(t,x)$, and $g, f(\cdot,\cdot,0,0)$ are bounded then $(f,g)$ satisfies Hypothesis \ref{Hpq}.

\subsection{Pseudo-Differential operators and Fractional Laplacian}\label{S4b}

This section  concerns  pseudo-differential operators with negative definite
 symbol, see \cite{jacob2} for an extensive description. A typical example of such 
operators will be the fractional Laplacian $(-\Delta)^{\frac{\alpha}{2}}$ with $\alpha\in]0,2[$, see Chapter 3 in \cite{di2012hitchhikers} for a detailed study of this operator. We will mainly use the notations and vocabulary of N. Jacob in
 \cite{jacob1}, \cite{jacob2} and \cite{jacob2005pseudo}, some results being attributed to W. Hoh \cite{hoh}. We fix $d\in\mathbbm{N}^*$.
$\mathcal{C}^{\infty}_c(\mathbbm{R}^d)$  will denote the space of real functions defined on
 $\mathbbm{R}^d$ which are infinitely continuously differentiable
 with compact support and $\mathcal{S}(\mathbbm{R}^d)$ the Schwartz space
of fast decreasing real smooth functions also defined on  $\mathbbm{R}^d$.
  $\mathcal{F}u$ will  denote the Fourier transform of a function $u$ whenever it is well-defined. For  $u\in L^1(\mathbbm{R}^d)$ we use the convention $\mathcal{F}u(\xi)=\frac{1}{(2\pi)^{\frac{d}{2}}}\int_{\mathbbm{R}^d} e^{-i(x,\xi)}u(x)dx$.

\begin{definition}
 A function $\psi\in\mathcal{C}(\mathbbm{R}^d,\mathbbm{R})$
 will be said 
\textbf{negative definite} if for any $k\in\mathbbm{N}$, $\xi_1,\cdots,\xi_k\in\mathbbm{R}^d$, the matrix 
$(\psi(\xi^j)+\psi(\xi^l)-\psi(\xi^j-\xi^l))_{j,l=1,\cdots,k}$ is  symmetric positive definite.
	\\
	\\
	A function $q\in\mathcal{C}(\mathbbm{R}^d\times \mathbbm{R}^d,\mathbbm{R})$ will be called a \textbf{continuous negative definite symbol} if for any $x\in\mathbbm{R}^d$, $q(x,\cdot)$ is continuous negative definite
	\\
 In this case we introduce
  the pseudo-differential operator $q(\cdot,D)$ defined by 
	\begin{equation}
		q(\cdot,D)(u)(x)=\frac{1}{(2\pi)^{\frac{d}{2}}}\int_{\mathbbm{R}^d} e^{i(x,\xi)}q(x,\xi)\mathcal{F}u(\xi)d\xi.
	\end{equation}
\end{definition}

\begin{remark}
By Theorem 4.5.7 in \cite{jacob1}, 
 $q(\cdot,D)$ maps the space 
 $\mathcal{C}^{\infty}_c(\mathbbm{R}^d)$ 
of smooth functions with compact support 
into itself.
In particular $q(\cdot,D)$ will be defined on
 $\mathcal{C}^{\infty}_c(\mathbbm{R}^d)$. However, the proof of this Theorem 4.5.7 only uses the fact that if $\phi\in\mathcal{C}^{\infty}_c(\mathbbm{R}^d)$ then $\mathcal{F}\phi\in\mathcal{S}(\mathbbm{R}^d)$ and this still holds for 
every $\phi\in\mathcal{S}(\mathbbm{R}^d)$. Therefore $q(\cdot,D)$ is well-defined on $\mathcal{S}(\mathbbm{R}^d)$ and maps it into 
$\mathcal{C}(\mathbbm{R}^d,\mathbbm{R})$.
\end{remark}

A typical example of such pseudo-differential operators is the fractional
 Laplacian defined for some fixed  $\alpha\in]0,2[$ on $\mathcal{S}(\mathbbm{R}^d)$ by 
\begin{equation}
(-\Delta)^{\frac{\alpha}{2}}(u)(x)=\frac{1}{(2\pi)^{\frac{d}{2}}}
\int_{\mathbbm{R}^d} e^{i(x,\xi)}\|\xi\|^{\alpha}\mathcal{F}u(\xi)d\xi.
\end{equation}
Its symbol has no dependence in $x$ and is the continuous negative definite function $\xi\mapsto \|\xi\|^{\alpha}$.
Combining Theorem 4.5.12 and 4.6.6  in \cite{jacob2005pseudo}, one can  state the following.
\begin{theorem}\label{ThmJacob1}
	Let $\psi$ be a continuous negative definite function satisfying for some $r_0,c_0>0$: $\psi(\xi)\geq c_0\|\xi\|^{r_0}$ if $\|\xi\|\geq 1$. Let $M$ be the smallest integer strictly superior to $(\frac{d}{r_0}\vee 2)+d$. Let $q$ be a continuous negative symbol verifying, for some $c,c'>0$ and
 $\gamma:\mathbbm{R}^d\rightarrow \mathbbm{R}^*_+$,  the following items.
	\begin{itemize}
		\item $q(\cdot,0)=0$ and $\underset{x\in\mathbbm{R}^d}{\text{sup }}|q(x,\xi)|\underset{\xi\rightarrow 0}{\longrightarrow}0$;
		\item $q$ is $\mathcal{C}^{2M+1-d}$ in the first variable and for any $\beta\in\mathbbm{N}^d$ with $\|\beta\|\leq 2M+1-d$, $\|\partial^{\beta}_x q\|\leq c(1+\psi)$;
		\item $q(x,\xi)\geq \gamma(x)(1+\psi(x))$ if $x\in \mathbbm{R}^d$, $\|\xi\|\geq 1$;
		\item $q(x,\xi)\leq c'(1+\|\xi\|^2)$ for every $(x,\xi)$.
	\end{itemize}
	Then the homogeneous Martingale Problem associated to $(-q(\cdot,D),\mathcal{S}(\mathbbm{R}^d))$ is well-posed (see Definition \ref{MPhomogene}) and its solution $(\mathbbm{P}^x)_{x\in\mathbbm{R}^d}$ defines a homogeneous Markov class, see Notation \ref{HomogeneNonHomogene}.
\end{theorem}

We will now introduce the time-inhomogeneous domain which will be used to extend $\mathcal{D}(-q(\cdot,D))=\mathcal{S}(\mathbbm{R}^d)$.


\begin{definition}\label{Ctopo}
	 We will denote by $\mathcal{C}^1([0,T],\mathcal{S}(\mathbbm{R}^d))$ the set of
	functions $\phi\in\mathcal{C}([0,T],\mathcal{S}(\mathbbm{R}^d))$ such that there exists a function
	$\partial_t\phi\in\mathcal{C}([0,T],\mathcal{S}(\mathbbm{R}^d))$ verifying the following. For every $t_0 \in [0,T]$ 
	we have  $\frac{1}{(t - t_0)}(\phi(t)-\phi(t_0))\underset{t\rightarrow t_0}{\overset{\mathcal{S}(\mathbbm{R}^d)}{\longrightarrow}}\partial_t\phi(t_0)$.
\end{definition}

We recall that $\mathcal{S}(\mathbbm{R}^d)$ is a  topological algebra, meaning  that addition, multiplication and multiplication by a scalar are continuous for its topology.

\begin{lemma}\label{LemmaS1}
For any $\phi,\psi\in\mathcal{C}^1([0,T],\mathcal{S}(\mathbbm{R}^d))$, we have $\partial_t(\phi\psi)=\psi\partial_t\phi+\phi\partial_t\psi$.
\end{lemma}
\begin{proof}
The proof is very close to  the one in $\mathbbm{R}$.
\end{proof}

\begin{notation}\label{DefS1}
We set $\mathcal{D}(\partial_t-q(\cdot,D)):=\mathcal{C}^1([0,T],\mathcal{S}(\mathbbm{R}^d))$.
\end{notation}
Elements in $\mathcal{C}([0,T],\mathcal{S}(\mathbbm{R}^d))$ will also be seen as functions of two variables, and since convergence in $\mathcal{S}(\mathbbm{R}^d)$ implies pointwise convergence, the usual notion of partial derivative coincides with the notation $\partial_t$ introduced in Definition \ref{Ctopo}. Any $\phi\in\mathcal{D}(\partial_t-q(\cdot,D))$ clearly verifies
\begin{itemize}
\item $\forall t\in[0,T]$, $\phi(t,\cdot)\in\mathcal{S}(\mathbbm{R}^d)$ and $\forall x\in\mathbbm{R}^d$, $\phi(\cdot,x)\in\mathcal{C}^1([0,T],\mathbbm{R})$;
\item $\forall t\in[0,T]$, $\partial_t\phi(t,\cdot)\in\mathcal{S}(\mathbbm{R}^d)$.
\end{itemize}
Our goal now is to show that $\mathcal{D}(\partial_t-q(\cdot,D)$ also verifies the other items needed to be included in $\mathcal{D}^{max}(\partial_t-q(\cdot,D))$ (see Notation \ref{NotDomain}) and therefore that Corollary \ref{conclusionA4} applies with this domain.

\begin{notation}
Let $\alpha,\beta\in\mathbbm{N}^d$ be multi-indices, we introduce the semi-norm
\begin{equation}
\|\cdot\|_{\alpha,\beta}:
\begin{array}{rcl}
\mathcal{S}(\mathbbm{R}^d)&\longrightarrow &\mathbbm{R}\\
\phi&\longmapsto &\underset{x\in\mathbbm{R}^d}{\text{sup }}|x^{\alpha}\partial_x^{\beta}\phi(x)|.
\end{array}
\end{equation}
\end{notation}
$\mathcal{S}(\mathbbm{R}^d)$ is a Fr\'echet space whose topology is determined
by the family of seminorms $ \|\cdot\|_{\alpha,\beta}$.
In particular those seminorms are continuous.
\\
In what follows, $\mathcal{F}_x$ will denote the Fourier transform taken in the space variable.
\begin{proposition}\label{PropS0}
Let $\phi\in\mathcal{C}([0,T],\mathcal{S}(\mathbbm{R}^d))$. 
Then $\mathcal{F}_x\phi\in\mathcal{C}([0,T],\mathcal{S}(\mathbbm{R}^d))$. Moreover if $\phi\in\mathcal{C}^1([0,T],\mathcal{S}(\mathbbm{R}^d))$, then $\mathcal{F}_x\phi\in\mathcal{C}^1([0,T],\mathcal{S}(\mathbbm{R}^d))$ and 
\\
$\partial_t\mathcal{F}_x\phi=\mathcal{F}_x\partial_t\phi$.
\end{proposition}
\begin{proof}
$\mathcal{F}_x:\mathcal{S}(\mathbbm{R}^d)\longrightarrow \mathcal{S}(\mathbbm{R}^d)$ is continuous, so $\phi\in\mathcal{C}([0,T],\mathcal{S}(\mathbbm{R}^d)$ implies $\mathcal{F}_x\phi\in\mathcal{C}([0,T],\mathcal{S}(\mathbbm{R}^d))$. If $\phi\in\mathcal{C}^1([0,T],\mathcal{S}(\mathbbm{R}^d))$ then $\partial_t\phi\in\mathcal{C}([0,T],\mathcal{S}(\mathbbm{R}^d)$ so $\mathcal{F}_x\partial_t\phi\in\mathcal{C}([0,T],\mathcal{S}(\mathbbm{R}^d)$. Then for any $t_0\in[0,T]$, the convergence  
\\
$\frac{1}{t-t_0}(\phi(t,\cdot)-\phi(t_0,\cdot))\underset{t\rightarrow t_0}{\overset{\mathcal{S}(\mathbbm{R}^d)}{\longrightarrow}}\partial_t\phi(t_0,\cdot)$ is preserved by the continuous mapping $\mathcal{F}_x$ meaning  that (by linearity)
\\
$\frac{1}{t-t_0}(\mathcal{F}_x\phi(t,\cdot)-\mathcal{F}_x\phi(t_0,\cdot))\underset{t\rightarrow t_0}{\overset{\mathcal{S}(\mathbbm{R}^d)}{\longrightarrow}}\mathcal{F}_x\partial_t\phi(t_0,\cdot)$. Since $\mathcal{F}_x\partial_t\phi\in\mathcal{C}([0,T],\mathcal{S}(\mathbbm{R}^d))$, we have shown that $\mathcal{F}_x\phi\in\mathcal{C}^1([0,T],\mathcal{S}(\mathbbm{R}^d))$ and $\partial_t\mathcal{F}_x\phi=\mathcal{F}_x\partial_t\phi$.
\end{proof}

\begin{proposition}\label{PropS1}
If $\phi\in\mathcal{C}([0,T],\mathcal{S}(\mathbbm{R}^d))$, then  for any $\alpha,\beta\in\mathbbm{N}^d$, 
\\
$(t,x)\longmapsto x^{\alpha}\partial_x^{\beta}\phi(t,x)$ 
is bounded. 
\end{proposition}
\begin{proof}
Let $\alpha,\beta$ be fixed. Since the maps 
$\|\cdot\|_{\alpha,\beta}:\mathcal{S}(\mathbbm{R}^d)\rightarrow\mathbbm{R}$ are
 continuous, for every  $\phi\in\mathcal{C}([0,T],\mathcal{S}(\mathbbm{R}^d))$,
 the application $t\mapsto \|\phi(t,\cdot)\|_{\alpha,\beta}$ is continuous on 
the compact interval $[0,T]$ and therefore bounded,
which yields the result.

\end{proof}

\begin{proposition}\label{PropS2}
If $\phi\in\mathcal{C}([0,T],\mathcal{S}(\mathbbm{R}^d))$ and $\alpha,\beta\in\mathbbm{N}^d$, then 
there exist non-negative functions $\psi_{\alpha,\beta}\in L^1(\mathbbm{R}^d)$ such that
 for every $(t,x)\in [0,T]\times \mathbbm{R}^d$, $|x^{\alpha}\partial_x^{\beta}\phi(t,x)|\leq 
\psi_{\alpha,\beta}(x)$.
\end{proposition}
\begin{proof}
We decompose 
\begin{equation}
\begin{array}{rcl}
|x^{\alpha}\partial_x^{\beta}\phi(t,x)| &=& |x^{\alpha}\partial_x^{\beta}\phi(t,x)|\mathds{1}_{[-1,1]^d}(x)+|x^{\alpha+(2,\cdots,2)}\partial_x^{\beta}\phi(t,x)|\frac{1}{\underset{i\leq d}{\Pi}x_i^2}\mathds{1}_{\mathbbm{R}^d\backslash[-1,1]^d}(x)\\
&\leq& C(\mathds{1}_{[-1,1]^d}(x)+\frac{1}{\underset{i\leq d}{\Pi}x_i^2}\mathds{1}_{\mathbbm{R}^d\backslash[-1,1]^d}(x)),
\end{array}
\end{equation}
where $C$ is some constant which exists  thanks to Proposition \ref{PropS1}.
\end{proof}

\begin{proposition}\label{PropS3}
Let $q$ be a continuous negative definite symbol verifying the assumptions of Theorem \ref{ThmJacob1} and let $\phi\in\mathcal{C}^1([0,T],\mathcal{S}(\mathbbm{R}^d))$. Then for any $x\in\mathbbm{R}^d$, 
$t\mapsto q(\cdot,D)\phi(t,x)\in\mathcal{C}^1([0,T],\mathbbm{R})$ and $\partial_tq(\cdot,D)\phi=q(\cdot,D)\partial_t\phi$.
\end{proposition}
\begin{proof}
We fix $\phi\in\mathcal{C}^1([0,T],\mathcal{S}(\mathbbm{R}^d))$, and $x\in\mathbbm{R}^d$. We wish to show that for any $\xi\in\mathbbm{R}^d$, $t\longmapsto \frac{1}{(2\pi)^{\frac{d}{2}}}\int_{\mathbbm{R}^d} e^{i(x,\xi)}q(x,\xi)\mathcal{F}_x\phi(t,\xi)d\xi$ is $\mathcal{C}^1$ with derivative \\ 
$ t\longmapsto \frac{1}{(2\pi)^{\frac{d}{2}}}\int_{\mathbbm{R}^d} e^{i(x,\xi)}q(x,\xi)\mathcal{F}_x\partial_t\phi(t,\xi)d\xi. $
Since $\phi\in\mathcal{C}^1([0,T],\mathcal{S}(\mathbbm{R}^d))$, then $\partial_t\phi\in\mathcal{C}([0,T],\mathcal{S}(\mathbbm{R}^d))$ and by Proposition \ref{PropS0}, $\mathcal{F}_x\partial_t\phi\in\mathcal{C}([0,T],\mathcal{S}(\mathbbm{R}^d))$.
Moreover since $q$ verifies the assumptions of Theorem \ref{ThmJacob1}, then $|q(x,\xi)|$ is bounded by $c'(1+\|\xi\|^2)$ for some constant $c'$. 
Therefore by Proposition \ref{PropS2}, there exists a non-negative $\psi\in L^1(\mathbbm{R}^d)$ such that for every $t,\xi$, $|q(x,\xi)\mathcal{F}_x\partial_t\phi(t,\xi)|\leq \psi(\xi)$. Since by Proposition \ref{PropS0}, $\mathcal{F}_x\partial_t\phi=\partial_t\mathcal{F}_x\phi$, this implies that for any $(t,\xi)$, $|\partial_te^{i(x,\xi)}q(x,\xi)\mathcal{F}_x\phi(t,\xi)|\leq \psi(\xi)$. So by the 
 theorem about the differentiation of integrals depending on a parameter, 
 for any $\xi\in\mathbbm{R}^d$, $t\longmapsto \frac{1}{(2\pi)^{\frac{d}{2}}}\int_{\mathbbm{R}^d} e^{i(x,\xi)}q(x,\xi)\mathcal{F}_x\phi(t,\xi)d\xi$ is of class 
 $\mathcal{C}^1$ with derivative \\
 $t\longmapsto \frac{1}{(2\pi)^{\frac{d}{2}}}\int_{\mathbbm{R}^d} e^{i(x,\xi)}q(x,\xi)\mathcal{F}_x\partial_t\phi(t,\xi)d\xi$.

\end{proof}

\begin{proposition}\label{PropS4}
Let $q$ be a continuous negative definite symbol verifying the assumptions of Theorem \ref{ThmJacob1} and let $\phi\in\mathcal{C}^1([0,T],\mathcal{S}(\mathbbm{R}^d))$. Then $\phi$, $\partial_t\phi$, $q(\cdot,D)\phi$ and $q(\cdot,D)\partial_t\phi$ are bounded.
\end{proposition}
\begin{proof}
Proposition \ref{PropS1} implies that any element of $\mathcal{C}([0,T],\mathcal{S}(\mathbbm{R}^d))$ is bounded, so we immediately deduce that $\phi$ and $\partial_t\phi$ are bounded.
\\
 Since $q$ verifies the assumptions of Theorem \ref{ThmJacob1},
  for any fixed $(t,x)\in[0,T]\times \mathbbm{R}^d$,  we have
\begin{equation}
\begin{array}{rcl}
|q(\cdot,D)\phi(t,x)|&=& \left|\frac{1}{(2\pi)^{\frac{d}{2}}}\int_{\mathbbm{R}^d} e^{i(x,\xi)}q(x,\xi)\mathcal{F}_x\phi(t,\xi)d\xi\right|\\
&\leq& C\int_{\mathbbm{R}^d}(1+\|\xi\|^2)|\mathcal{F}_x\phi(t,\xi)|d\xi,
\end{array}
\end{equation}
for some constant $C$.
 Since $\phi\in\mathcal{C}([0,T],\mathcal{S}(\mathbbm{R}^d))$ then,
 by Proposition \ref{PropS0},
 $\mathcal{F}_x\phi$ also belongs to 
 $\mathcal{C}([0,T],\mathcal{S}(\mathbbm{R}^d))$, and by Proposition \ref{PropS1}, there exists a positive $\psi\in L^1(\mathbbm{R}^d)$ such that for any $(t,\xi)$, $(1+\|\xi\|^2)|\mathcal{F}_x\phi(t,\xi)|\leq \psi(\xi)$, so for any $(t,x)$, $|q(\cdot,D)\phi(t,x)|\leq \|\psi\|_1$.
\\
\\
Similar arguments hold replacing $\phi$ with $\partial_t\phi$ since it also belongs to $\mathcal{C}([0,T],\mathcal{S}(\mathbbm{R}^d))$.
\end{proof}

%

\begin{remark} \label{RSobolev}
  $\mathcal{C}^1([0,T],\mathcal{S}(\mathbbm{R}^d))$ seems
to be a domain which is particularly appropriate  for time-dependent Fourier 
analysis and it fits well for our framework. 
 On the other hand it is not so fundamental to require such regularity
for classical solutions for Pseudo-PDEs, so that we could
consider a larger domain. For example  
the Fr\'echet algebra $\mathcal{S}(\mathbbm{R}^d)$ could be replaced
with the  Banach algebra $W^{d+3,1}(\mathbbm{R}^d)\bigcap W^{d+3,\infty}(\mathbbm{R}^d)$
in all the previous proofs.
\\
\end{remark}

\begin{corollary}\label{CoroS1}
Let $q$ be a continuous negative definite symbol verifying the hypotheses of Theorem \ref{ThmJacob1}. Then the properties below are valid.
$\mathcal{D}(\partial_t-q(\cdot,D))$
 is a linear algebra included in $\mathcal{D}^{max}(\partial_t-q(\cdot,D))$ as defined in Notation \ref{NotDomain}.
\end{corollary}
\begin{proof}
We recall that, according to Notation \ref{DefS1}
$\mathcal{D}(\partial_t-q(\cdot,D))
=\mathcal{C}^1([0,T],\mathcal{S}(\mathbbm{R}^d))$.
The proof follows from Lemma \ref{LemmaS1}, Propositions \ref{PropS3} and \ref{PropS4}, and the comments under Notation \ref{DefS1}.
\end{proof}

\begin{corollary}\label{CoroJacob}
Let $q$ be a continuous negative definite symbol verifying the hypotheses of Theorem \ref{ThmJacob1}, let $(\mathbbm{P}^x)_{x\in\mathbbm{R}^d}$ be the corresponding homogeneous Markov class exhibited in Theorem \ref{ThmJacob1}, let $(\mathbbm{P}^{s,x})_{(s,x)\in[0,T]\times \mathbbm{R}^d}$ be the corresponding Markov class (see Notation \ref{HomogeneNonHomogene}), let $(\mathcal{D}(\partial_t-q(\cdot,D)),\partial_t-q(\cdot,D))$ be as in Notation \ref{DefS1}.
 Then 
 \begin{itemize}
 	\item $(\mathbbm{P}^{s,x})_{(s,x)\in[0,T]\times \mathbbm{R}^d}$ solves the well-posed Martingale Problem associated to $(\mathcal{D}(\partial_t-q(\cdot,D)),\partial_t-q(\cdot,D))$;
 	\item its transition function is measurable in time.
 \end{itemize}
\end{corollary}
\begin{proof}
	The first statement directly comes from Theorem \ref{ThmJacob1} and Corollaries \ref{CoroS1} \ref{conclusionA4}, and the second from Proposition \ref{HomoMeasurableint}.
\end{proof}

\begin{remark}
	The symbol of the fractional Laplacian $q:(x,\xi)\mapsto\|\xi\|^{\alpha}$ trivially verifies the assumptions of Theorem \ref{ThmJacob1}. Indeed, it has no dependence in $x$, so it is enough to set $\psi:\xi\mapsto \|\xi\|^{\alpha}$, $c_0=c=c'=1$, $r_0=\alpha$ and $\gamma=\frac{1}{2}$. 
\end{remark}

The Pseudo-PDE that we focus on is the following.
\begin{equation}\label{PDEsymbol}
\left\{
\begin{array}{rcl}
\partial_tu-q(\cdot,D)u = f(\cdot,\cdot,u,\Gamma(u)^{\frac{1}{2}}) \,\text{on }[0,T]\times\mathbbm{R}^d\\
u(T,\cdot)=g,
\end{array}\right.
\end{equation}
where $q$ is a continuous negative definite symbol verifying the assumptions of Theorem \ref{ThmJacob1} and $\Gamma$ is the associated carr\'e du champs operator, see Definition \ref{SFO}.

\begin{remark} \label{R418}
	By Proposition 3.3 in \cite{di2012hitchhikers}, for any  $\alpha\in]0,2[$, there exists a constant $c_{\alpha}$ such that for any  $\phi\in\mathcal{S}(\mathbbm{R}^d)$, 
	\begin{equation}\label{FracLap}
	(-\Delta)^{\frac{\alpha}{2}}\phi  = c_{\alpha}PV\int_{\mathbbm{R}^d} \frac{(\phi(\cdot+y)-\phi)}{\| y\|^{d+\alpha}}dy,
	\end{equation}
	where  $PV$ is a notation for principal value, see (3.1) in \cite{di2012hitchhikers}. Therefore in the particular case of the fractional Laplace operator, the carr\'e du champs operator associated to  $(-\Delta)^{\frac{\alpha}{2}}$ is given by 
	\begin{equation}
	\begin{array}{rcl}
	&&\Gamma_{\alpha}(\phi)\\
	&=& c_{\alpha}PV\int_{\mathbbm{R}^d} \frac{(\phi^2(\cdot,\cdot+y)-\phi^2)}{\| y\|^{d+\alpha}}dy-2\phi c_{\alpha}PV\int_{\mathbbm{R}^d} \frac{(\phi(\cdot,\cdot+y)-\phi)}{\| y\|^{d+\alpha}}dy\\
	&=&c_{\alpha}PV\int_{\mathbbm{R}^d} \frac{(\phi(\cdot,\cdot+y)-\phi)^2}{\| y\|^{d+\alpha}}dy.
	\end{array}
	\end{equation}
\end{remark}

\begin{proposition}
	Let  $q$ be a continuous negative symbol verifying the assumptions of Theorem \ref{ThmJacob1}, let $(\mathbbm{P}^{s,x})_{(s,x)\in[0,T]\times \mathbbm{R}^d}$ be the Markov class which by Corollary \ref{CoroJacob} solves the well-posed Martingale Problem associated to
	\\
	 $(\mathcal{D}(\partial_t-q(\cdot,D)),\partial_t-q(\cdot,D))$. 
	
	For any $(f,g)$ verifying  \ref{Hpq}, $Pseudo-PDE(f,g)$ admits a unique decoupled mild solution in the sense of Definition \ref{mildsoluv}.
	
\end{proposition}
\begin{proof}
	The assertion comes from Corollary \ref{CoroJacob} and Theorem \ref{MainTheorem}.
\end{proof}

\subsection{Parabolic semi-linear PDEs with distributional drift}\label{S4c}

In this section we will use the formalism and results obtained 
 in \cite{frw1} and \cite{frw2}, see also \cite{russo_trutnau07}, \cite{diel} 
for more recent developments. In particular the latter
paper treats interesting applications to polymers.
Those papers introduced a suitable framework of 
 Martingale Problem related to a PDE operator 
containing a distributional drift $b'$ which is the
derivative of a continuous function. 
\cite{issoglio} established a first work in the $n$-dimensional setting.

Let $b,\sigma\in \mathcal{C}^0(\mathbbm{R})$ such that $\sigma>0$.
 By  mollifier, we intend a function $\Phi\in\mathcal{S}(\mathbbm{R})$ with $\int \Phi(x)dx=1$. 
We denote $\Phi_n(x)=n\Phi(nx)$, $\sigma^2_n = \sigma^2\ast \Phi_n$, $b_n=b\ast \Phi_n$.
We then define $L_ng=\frac{\sigma^2_n}{2}g''+b_n'g'$.
$f\in\mathcal{C}^1(\mathbbm{R})$ is said to be a solution to $Lf =\dot{l}$ where $\dot{l}\in\mathcal{C}^0$, if for any mollifier $\Phi$, there are sequences $(f_n)$ in $\mathcal{C}^2$, $(\dot{l}_n)$ in $\mathcal{C}^0$ such that $L_nf_n=(\dot{l}_n)\text{,  }f_n\overset{\mathcal{C}^1}{\longrightarrow}f\text{,  }\dot{l}_n\overset{\mathcal{C}^0}{\longrightarrow}\dot{l}$.
We will assume that 
$\Sigma(x) =\underset{n\rightarrow\infty}{\text{lim }}2\int_0^x\frac{b'_n}{\sigma^2_n}(y)dy$
exists in $\mathcal{C}^0$ independently from the mollifier.

By Proposition 2.3 in \cite{frw1} there exists a solution $h\in\mathcal{C}^1$ to $Lh=0\text{,  }h(0)=0\text{,  }h'(0)=1$.
 Moreover it verifies $h'=e^{-\Sigma}$. Moreover by Remark 2.4 in 
 \cite{frw1}, 
for any $\dot{l}\in\mathcal{C}^0$, $x_0,x_1\in\mathbbm{R}$,   
there exists a unique solution of
\begin{equation}
Lf(x)=\dot{l}\text{,  }f\in\mathcal{C}^1\text{,  }f(0)=x_0\text{,  }f'(0)=x_1.
\end{equation}
$\mathcal{D}_L$ is defined as the set of $f\in\mathcal{C}^1$ such that there exists some $\dot{l}\in\mathcal{C}^0$ with $Lf=\dot{l}$. And by Lemma 2.9 in \cite{frw1} it is equal to the set of $f\in\mathcal{C}^1$ such that  $\frac{f'}{h'}\in\mathcal{C}^1$. So it is clearly an algebra.
$h$ is strictly increasing,  $I$ will denote its image. Let $L^0$ be the classical differential operator defined by $L^0\phi=\frac{\sigma^2_0}{2}\phi''$,
where
\begin{equation}
\sigma_0(y)=\left\{
\begin{array}{rcl}
(\sigma h')(h^{-1}(y)) & : & y\in I\\
0 & : & y\in I^c.
\end{array}\right.
\end{equation}
\\
Let $v$ be the unique solution to $Lv=1\text{,  }v(0)=v'(0)=0$, we will assume that
\begin{equation}\label{NonExplosion}
v(-\infty)=v(+\infty)=+\infty,
\end{equation}
which represents a non-explosion condition.
In this case, Proposition 3.13 in \cite{frw1} states that the Martingale
 Problem associated to 
$(\mathcal{D}_L,L)$ is well-posed. Its solution will be denoted $(\mathbbm{P}^{s,x})_{(s,x)\in[0,T]\times\mathbbm{R}^d}$.
By Proposition 2.13, $\mathcal{D}_{L^0} = C^2(I)$.
 and by Proposition 3.2 in \cite{frw1}, the Martingale Problem associated to
 $(\mathcal{D}_{L^0},L^0)$ 
is also well-posed, we will call $(\mathbbm{Q}^{s,x})_{(s,x)\in[0,T]\times\mathbbm{R}^d}$ its solution. Moreover under any $\mathbbm{P}^{s,x}$ the canonical process is a Dirichlet process, and $h^{-1}(X)$ is a semi-martingale that we call $Y$ 
 solving the SDE $Y_t=h(x)+\int_s^t\sigma_0(Y_s)dW_s$ in law, where the law of $Y$ is  $\mathbbm{Q}^{s,x}$. $X_t$ is a $\mathbbm{P}^{s,x}$-Dirichlet process
 whose martingale component is $\int_s^{\cdot}\sigma(X_r)dW_r$. 
$(\mathbbm{P}^{s,x})_{(s,x)\in[0,T]\times\mathbbm{R}^d}$ and $(\mathbbm{Q}^{s,x})_{(s,x)\in[0,T]\times\mathbbm{R}^d}$ both define Markov classes.

We introduce now the domain that we will indeed use.
\begin{definition}\label{domain}
We set 
\begin{equation}
\mathcal{D}(a)=\left\{
\phi\in\mathcal{C}^{1,1}([0,T]\times\mathbbm{R}):\frac{\partial_x\phi}{h'}\in\mathcal{C}^{1,1}([0,T]\times\mathbbm{R})\right\},
\end{equation}
which clearly is a linear algebra.
On $\mathcal{D}(a)$, we set $L\phi:=\frac{\sigma^2h'}{2}\partial_x(\frac{\partial_x\phi}{h'})$ and $a(\phi):=\partial_t\phi+L\phi$.
\end{definition}

\begin{proposition}
	Let $\Gamma$ denote the carré du champ operator associated to $a$, let $\phi,\psi$ be in $\mathcal{D}(a)$, then $\Gamma(\phi,\psi)=\sigma^2\partial_x\phi\partial_x\psi$.
\end{proposition}
\begin{proof}
	We fix $\phi,\psi$ in $\mathcal{D}(a)$. We write
	\begin{equation}
		\begin{array}{rcl}
		\Gamma(\phi,\psi)&=& (\partial_t+L)(\phi\psi)-\phi(\partial_t+L)(\psi)-\psi(\partial_t+L)(\phi)\\
		&=& \frac{\sigma^2h'}{2}\left(\partial_x\left(\frac{\partial_x\phi\psi}{h'}\right)-\phi\partial_x\left(\frac{\partial_x\psi}{h'}\right)-\psi\partial_x\left(\frac{\partial_x\phi}{h'}\right)\right)\\
		&=& \sigma^2\partial_x\phi\partial_x\psi.
		\end{array}
	\end{equation}
\end{proof}
Emphasizing  that $b'$ is a distribution,
the equation that we will  study in this section is therefore given by
\begin{equation}\label{PDEdistri}
\left\{
\begin{array}{l}
 \partial_tu + \frac{1}{2}\sigma^2\partial^2_x u + b'\partial_xu +f(\cdot,\cdot,u,\sigma|\partial_xu|)=0\quad\text{ on }[0,T]\times\mathbbm{R}\\
 u(T,\cdot) = g.
\end{array}\right.
\end{equation}

\begin{proposition}\label{MPnewdomaindistri}
$(\mathbbm{P}^{s,x})_{(s,x)\in[0,T]\times\mathbbm{R}^d}$  solves the Martingale Problem associated to $(a,\mathcal{D}(a))$.
\end{proposition}

\begin{proof}
 $(t,y)\mapsto   \phi(t,h^{-1}(y))$ is of class $\mathcal{C}^{1,2}$; moreover
$\partial_x\left(\phi(r,\cdot)\circ h^{-1}\right)=\frac{\partial_x\phi}{h'}\circ h^{-1}$ and 
$\partial_x^2\left(\phi(r,\cdot)\circ h^{-1}\right)=\frac{2L\phi}{\sigma^2h'^2}\circ h^{-1}=\frac{2L\phi}{\sigma_0^2}\circ h^{-1}$. 
 By It\^o formula we have
\begin{equation}
    \begin{array}{rcl}
     \phi(t,X_t)&=&\phi(t,h^{-1}(Y_t))\\
     &=&\phi(s,x)+\int_s^t\left(\partial_t\phi(r,h^{-1}(Y_r))+\frac{1}{2}\sigma_0^2(Y_r)\partial_x^2\left(\phi(r,\cdot)\circ h^{-1}\right)(Y_r)\right)dr\\
      &&+ \int_s^t\sigma_0(r,h^{-1}(Y_r))\partial_x\left(\phi(r,\cdot)\circ h^{-1}\right)(Y_r)dW_r\\
     &=&\phi(s,x)+\int_s^t\left(\partial_t\phi(r,h^{-1}(Y_r))+L\phi(r,h^{-1}(Y_r))\right)dr\\
     &&+\int_s^t\sigma_0(r,h^{-1}(Y_r))\frac{\partial_x\phi(r,h^{-1}(Y_r))}{h'(Y_r)}dW_r\\
    &=&\phi(s,x)+\int_s^t\left(\partial_t\phi(r,X_r)+l(r,X_r))\right)dr+\int_s^t\sigma(r,X_r)\partial_x\phi(r,X_r)dW_r.
    \end{array}
\end{equation}
 
Therefore $\phi(t,X_t)-\phi(s,x)-\int_s^ta(\phi)(r,X_r)dr=\int_s^t\sigma(r,X_r)\partial_x\phi(r,X_r)dW_r$ is a  local martingale.
\end{proof}

In order to consider the $FBSDE^{s,x}(f,g)$ for functions $(f,g)$  having polynomial growth in $x$ we will show the following result.
We formulate here the supplementary assumption, called (TA) in \cite{frw1}.
This means the  existence of strictly positive constants  $c_1, C_1$ such that
\begin{equation} \label{(TA)}
c_1 \le \frac{e^\Sigma}{\sigma} \le C_1.
\end{equation}
\begin{proposition}\label{MomentsDistri}
We suppose that \eqref{(TA)} is fulfilled and $\sigma$ has linear growth.
Then, for any  $p > 0$ and   $(s,x)\in[0,T]\times\mathbbm{R}$, $\mathbbm{E}^{s,x}[|X_T|^p]<\infty$ and 
$\mathbbm{E}^{s,x}[\int_s^T|X_r|^p dr]<\infty$. 

\end{proposition}

\begin{proof}

We start by proving the proposition in the divergence form case,
 meaning that $b=\frac{\sigma^2}{2}$.
Let $(s,x)$ and $t\in[s,T]$ be fixed. Thanks to the Aronson estimates,
see e.g. \cite{ar} and also  Section 5. of \cite{frw1}, there is a 
constant $M > 0$ such that
\begin{equation}
\begin{array}{rcl}
 \mathbbm{E}^{s,x}[|X_t|^p] &=&\int_{\mathbbm{R}}|y|^pp_{t-s}(x,y)dy\\
 &\leq& \frac{M}{\sqrt{t-s}}\int_{\mathbbm{R}}|y|^pe^{-\frac{|x-y|^2}{M(t-s)}}dz\\
 &=&M^{\frac{3}{2}}\int_{\mathbbm{R}}|x+z\sqrt{M(t-s)}|^pe^{-z^2}dz\\
 &\leq&\sum_{k=0}^p M^{\frac{3+k}{2}} \binom{p}{k}
 |x|^k|t-s|^{\frac{p-k}{2}}\int_{\mathbbm{R}}|z|^{p-k}e^{-z^2}dz,
\end{array}
\end{equation}
which 
 (for fixed $(s,x)$) is   bounded in $t \in [s,T] $ 
and therefore Lebesgue integrable in $t$ on $[s,T]$. 
 This in particular shows that $\mathbbm{E}^{s,x}[|X_T|^p]$ and 
$ \mathbbm{E}^{s,x}[\int_s^T|X_r|^pdr] (=
\int_s^T\mathbbm{E}^{s,x}[|X_r|^p]dr)$ are finite.

Now we will consider the case in which $X$  only verifies \eqref{(TA)}  and we will add the hypothesis that $\sigma$ has linear growth.
Then there exists a process $Z$ (see Lemma 5.6 in \cite{frw1}) solving an SDE with distributional drift of divergence form generator, and a
 function $k$ of class  $\mathcal{C}^1$ such that $X=k^{-1}(Z)$. The
 \eqref{(TA)} condition implies that there exist two constants such that $0<c\leq k'\sigma \leq C$ implying that for any $x$,
$(k^{-1})'(x)=\frac{1}{k'\circ k^{-1}(x)}\leq \frac{\sigma\circ k^{-1}(x)}{c}\leq C_2(1+ |k^{-1}(x)|)$,
for a positive constant $C_2$. So by Gronwall Lemma there exists $C_3>0$ such
 that $k^{-1}(x)\leq C_3e^{C_2|x|},  \ \forall x \in {\mathbb R}$.
Now thank to the Aronson estimates  on  the transition function 
$p^Z$
of $Z$, for every  $p > 0$, we have
\begin{equation}
\begin{array}{rcl}
\mathbbm{E}^{s,x}[|X|_t^p] &\leq& C_3\int e^{C_2p|z|}p^Z_{t-s}(k(x),z)dz\\
&\leq& \int e^{C_2p|z|}\frac{M}{\sqrt{t}}e^{-\frac{|k(x)-z|^2}{Mt}}dz\\
&\leq& M^{\frac{3}{2}}\int e^{C_2p(\sqrt{Mt}|y|+k(x))}e^{-y^2}dy\\
&\leq& Ae^{B k(x)} ,
\end{array}
\end{equation}
where $A,B$ are two constants depending on $p$ and $M$.
This implies that \\
$\mathbbm{E}^{s,x}[|X_T|^p]<\infty$ and 
$\mathbbm{E}^{s,x}[\int_s^T|X_r|^pdr]<\infty$.
\end{proof}

We can now state the  main result of this section. 
\begin{proposition}
Assume that the non-explosion condition \eqref{NonExplosion} is verified, that $f$ is Lipschitz in $(y,z)$ uniformly in $(t,x)$ 
and the validity  of one of the two following items.
\begin{itemize}
	\item  the (TA) condition \eqref{(TA)} is fulfilled, $\sigma$ has linear growth and $g$ has polynomial growth and $f$ has polynomial growth in $x$ uniformly in $t$;
	\item $f(\cdot,\cdot,0,0)$ and $g$ are bounded.
\end{itemize} 
 Then \eqref{PDEdistri} has a unique decoupled mild solution $u$ in the sense of Definition \ref{mildsoluv}.
\end{proposition}
\begin{proof}
The assertion comes from Theorem \ref{MainTheorem} which applies thanks to Propositions \ref{MPnewdomaindistri}, \ref{MomentsDistri} and  \ref{HomoMeasurableint}. 
\end{proof}
\begin{remark} 
\begin{enumerate}
\item A first analysis linking 
PDEs (in fact second order elliptic differential equations) with distributional drift and BSDEs
was performed by \cite{wurzer}.
In those BSDEs the final horizon was a stopping time. 
\item In \cite{issoglio_jing16}, the authors  have considered a class of BSDEs
involving particular distributions.
\end{enumerate}
\end{remark}

\subsection{Diffusion equations on differential manifolds}\label{S4d}

In this section, we will provide an example of application in a non
 Euclidean space. 
We consider a compact connected smooth differential manifold $M$ of dimension $n$. We denote by $\mathcal{C}^{\infty}(M)$ the linear algebra of smooth functions from $M$ to $\mathbbm{R}$, and $(U_i,\phi_i)_{i\in I}$ its atlas. The reader may consult \cite{jost} 
for an extensive introduction to the study of differential manifolds, and \cite{hsu} concerning diffusions on differential manifolds.

\begin{lemma} \label{LPolish}
  	$M$ is Polish.
\end{lemma}

\begin{proof} \
By Theorem 1.4.1 in \cite{jost} $M$ may be equipped with a Riemannian metric,
that we denote by $m$ and its topology may be metricized by the associated distance which we denote by $d$. 
As any 
compact metric space, $(M,d)$ is separable and complete
 so that $M$ is  a Polish space.
\end{proof}
We denote by $(\Omega,\mathcal{F},(X_t)_{t\in[0,T]}(\mathcal{F})_{t\in[0,T]})$ the canonical space associated to $M$ and $T$, 
see Definition \ref{canonicalspace}.

\begin{definition}
An operator $L:\mathcal{C}^{\infty}(M)\longrightarrow\mathcal{C}^{\infty}(M)$ will be called a \textbf{smooth second order elliptic non degenerate differential operator on $M$} if for any chart 
$\phi:U\longrightarrow \mathbbm{R}^n$ there exist smooth $\mu:\phi(U)\longrightarrow \mathbbm{R}^n$  and $\alpha:\phi(U)\longrightarrow S^*_+(\mathbbm{R}^n)$  such that on $\phi(U)$ for any $f\in\mathcal{C}^{\infty}(M)$ we have
\begin{equation}
Lf(\phi^{-1}(x))=\frac{1}{2}\underset{i,j=1}{\overset{n}{\sum}}\alpha^{i,j}(x)\partial_{x_ix_j}(f\circ\phi^{-1})(x)+\underset{i=1}{\overset{n}{\sum}}\mu^i(x)\partial_{x_i}(f\circ\phi^{-1})(x).
\end{equation}
\end{definition}
$\alpha$ and $\mu$ depend on the chart $\phi$ but this dependence will be sometimes
 omitted and we will say that for some given local coordinates,
\\ $Lf=\frac{1}{2}\underset{i,j=1}{\overset{n}{\sum}}\alpha^{i,j}\partial_{x_ix_j}f+\underset{i=1}{\overset{n}{\sum}}\mu^i\partial_{x_i}f$.

The following definition comes from \cite{hsu}, see Definition 1.3.1.
\begin{definition}
 Let $L$ denote a smooth second order elliptic non degenerate differential operator on $M$.
 Let $x\in M$. A probability measure $\mathbbm{P}^x$ on $(\Omega,\mathcal{F})$ will be called an \textbf{$L$-diffusion starting in $x$} if
 \begin{itemize}
 	\item $\mathbbm{P}^x(X_0=x)=1$;
 	\item for every $f\in\mathcal{C}^{\infty}(M)$, $f(X)-\int_0^{\cdot}Lf(X_r)dr$ is a $(\mathbbm{P}^x,(\mathcal{F})_{t\in[0,T]}) )$ local martingale.
 \end{itemize}
\end{definition}

\begin{remark} 
No explosion can occur for continuous stochastic processes with values in 
a compact  space, so there is no need to consider paths in the compactification of $M$ as in Definition 1.1.4 in \cite{hsu}.
	
	Theorems 1.3.4 and 1.3.5 in \cite{hsu} state that for any $x\in M$, there exists a unique $L$-diffusion starting in $x$.  Theorem 1.3.7 in \cite{hsu} implies that those probability measures $(\mathbbm{P}^x)_{x\in M}$ define a homogeneous Markov class.
\end{remark}

For a given operator $L$, the carr\'e du champs operator $\Gamma$ is given (in local coordinates) by
$\Gamma(\phi,\psi)=\underset{i,j=1}{\overset{n}{\sum}}\alpha^{i,j}\partial_{x_i}\phi\partial_{x_j}\phi$,
see equation (1.3.3) in \cite{hsu}. We wish to emphasize here that the carr\'e du champs operator has recently become a powerful tool in the study of geometrical properties of Riemannian manifolds. The reader may refer e.g. to \cite{bakry}.

\begin{definition}
	$(\mathbbm{P}^x)_{x\in M}$ will be called the \textbf{$L$-diffusion}. 
	 If $M$ is equipped with a specific Riemannian metric $m$ 
and $L$ is chosen to be equal to $\frac{1}{2}\Delta_m$ where $\Delta_m$ the Laplace-Beltrami operator associated to $m$, then $(\mathbbm{P}^x)_{x\in M}$ will be called the \textbf{Brownian motion associated to $m$}, see \cite{hsu} Chapter 3 for details.
\end{definition}

We now fix some smooth second order elliptic non degenerate differential operator $L$ and the $L$-diffusion $(\mathbbm{P}^x)_{x\in M}$. We introduce the associated Markov class $(\mathbbm{P}^{s,x})_{(s,x)\in[0,T]\times M}$ as described in Notation \ref{HomogeneNonHomogene}, which by Proposition \ref{HomoMeasurableint} is measurable in time.

\begin{notation}
	We define $\mathcal{D}(\partial_t+L)$ the set of functions $u:[0,T]\times M\longrightarrow \mathbbm{R}$ such that, for any chart $\phi:U\longrightarrow \mathbbm{R}^n$, the mapping
	\begin{equation}
		\begin{array}{rcl}
		[0,T]\times \phi(U)&\longrightarrow& \mathbbm{R}\\
		(t,x) &\longmapsto& u(t,\phi^{-1}(x))
		\end{array}
	\end{equation}
	belongs to $\mathcal{C}^{\infty}([0,T]\times \phi(U),\mathbbm{R})$, the set of infinitely continuously differentiable functions in the usual Euclidean setup.
\end{notation}

\begin{lemma}\label{LemmaManifold}
	$\mathcal{D}(\partial_t+L)$ is a linear algebra included in $\mathcal{D}^{max}(\partial_t+L)$ as defined in Notation \ref{NotDomain}.
\end{lemma}
\begin{proof}
	For some fixed chart  $\phi:U\longrightarrow \mathbbm{R}^n$, $\mathcal{C}^{\infty}([0,T]\times \phi(U),\mathbbm{R})$ is an algebra, so it is immediate that $\mathcal{D}(\partial_t+L)$ is an algebra.
	Moreover, if $u\in\mathcal{D}(\partial_t+L)$, it is clear that 
	\begin{itemize}
		\item $\forall x\in M$, $u(\cdot,x)\in\mathcal{C}^1([0,T],\mathbbm{R})$ and $\forall t\in[0,T]$, $u(t,\cdot)\in\mathcal{C}^{\infty}(M)$,
		\item $\forall t\in[0,T]$, $\partial_t u(t,\cdot)\in\mathcal{C}^{\infty}(M)$ and $\forall x\in M$, $Lu(\cdot,x)\in\mathcal{C}^1([0,T],\mathbbm{R})$.
	\end{itemize}
Given a chart $\phi:U\longrightarrow \mathbbm{R}^n$, by the Schwarz
 Theorem allowing the commutation of partial derivatives
 (in the classical Euclidean setup), we have for $x\in\phi(U)$
\begin{equation}
\begin{array}{rcl}
  \partial_t\circ L(u)(t,\phi^{-1}(x))&=&\frac{1}{2}\underset{i,j=1}{\overset{n}{\sum}}\alpha^{i,j}(x)\partial_t\partial_{x_ix_j}(u(\cdot,\phi^{-1}(\cdot))(t,x)\\
                                          &+&\underset{i=1}{\overset{n}{\sum}}\mu^i(x)\partial_t\partial_{x_i}(u(\cdot,\phi^{-1}(\cdot))(t,x)\\
                                      &=&\frac{1}{2}\underset{i,j=1}{\overset{n}{\sum}}\alpha^{i,j}(x)\partial_{x_ix_j}\partial_t(u(\cdot,\phi^{-1}(\cdot))(t,x)\\
  &+&\underset{i=1}{\overset{n}{\sum}}\mu^i(x)\partial_{x_i}\partial_t(u(\cdot,\phi^{-1}(\cdot))(t,x)\\
&=&L\circ\partial_t (u)(t,\phi^{-1}(x)).
\end{array}
\end{equation}
So $\partial_t\circ Lu=L\circ\partial_tu$. Finally $\partial_tu$, $Lu$ and $\partial_t\circ Lu$ are continuous (since they are continuous on all the sets $[0,T]\times U$ where $U$ is the domain of a chart) and are therefore bounded as continuous functions on the compact set $[0,T]\times M$.
This concludes the proof.
\end{proof}

\begin{corollary}
	$(\mathbbm{P}^{s,x})_{(s,x)\in[0,T]\times M}$ solves the well-posed Martingale Problem associated to $(\partial_t+L,\mathcal{D}(\partial_t+L))$ in the sense of Definition \ref{MartingaleProblem}.
\end{corollary}
 
\begin{proof}
	The corollary derives from Lemma \ref{LemmaManifold} and Corollary \ref{conclusionA4}.
\end{proof}

We fix a couple of functions $(f,g)$ with $f$ Lipschitz in $(y,z)$ uniformly in $(t,x)$, and $g, f(\cdot,\cdot,0,0)$ bounded. 
We consider the PDE
\begin{equation}\label{PDEmanifold}
\left\{
\begin{array}{l}
 \partial_tu + Lu +f(\cdot,\cdot,u,(\alpha\nabla u\cdot\nabla u)^{\frac{1}{2}})=0\quad\text{ on }[0,T]\times M\\
 u(T,\cdot) = g.
\end{array}\right.
\end{equation}
Since  Theorem \ref{MainTheorem}  applies, we have the following result.
\begin{proposition}
Equation \eqref{PDEmanifold} admits a unique decoupled mild solution $u$ in the sense of Definition \ref{mildsoluv}. 
\end{proposition}

\begin{appendices}

\section{Markov classes}\label{A2}

In this Appendix we recall some basic definitions and results concerning Markov processes. For a complete study of homogeneous Markov processes, one may consult \cite{dellmeyerD}, concerning non-homogeneous Markov classes, our reference was chapter VI of \cite{dynkin1982markov}. The first definition refers to the  canonical space that one can find in \cite{jacod79}, see paragraph 12.63.
\begin{notation}\label{canonicalspace}
In the whole section  $E$ will be a fixed  Polish  space (a separable 
completely metrizable topological space). $E$ will be called the \textbf{state space}. 
\\
\\
We consider $T\in\mathbbm{R}^*_+$. We denote $\Omega:=\mathbbm{D}([0,T],E)$ 
the space of functions from $[0,T]$ to $E$  right-continuous  with left limits and continuous at time $T$, e.g. c\`adl`ag. For any $t\in[0,T]$ we denote the coordinate mapping $X_t:\omega\mapsto\omega(t)$, and we introduce on $\Omega$ the $\sigma$-field  $\mathcal{F}:=\sigma(X_r|r\in[0,T])$. 

On the measurable space $(\Omega,\mathcal{F})$, we introduce the measurable \textbf{canonical process}
\begin{equation*}
X:
\begin{array}{rcl}
(t,\omega)&\longmapsto& \omega(t)\\ \relax
([0,T]\times \Omega,\mathcal{B}([0,T])\otimes\mathcal{F}) &\longrightarrow & (E,\mathcal{B}(E)),
\end{array}
\end{equation*}
and the right-continuous filtration $(\mathcal{F}_t)_{t\in[0,T]}$ where
$\mathcal{F}_t:=\underset{s\in]t,T]}{\bigcap}\sigma(X_r|r\leq s),$ if $t<T$, and $\mathcal{F}_T:= \sigma(X_r|r\in[0,T])=\mathcal{F}$.
$\left(\Omega,\mathcal{F},(X_t)_{t\in[0,T]},(\mathcal{F}_t)_{t\in[0,T]}\right)$ will be called the \textbf{canonical space} (associated to $T$ and $E$).
For any $t \in [0,T]$ we denote $\mathcal{F}_{t,T}:=\sigma(X_r|r\geq t)$, and
for any $0\leq t\leq u<T$ we will denote
$\mathcal{F}_{t,u}:= \underset{n\geq 0}{\bigcap}\sigma(X_r|r\in[t,u+\frac{1}{n}])$.
\end{notation}
Since $E$ is Polish, we recall that $\mathbbm{D}([0,T],E)$ can be equipped with a Skorokhod distance which makes it a Polish metric space (see Theorem 5.6 in chapter 3 of \cite{EthierKurz}, and for which the Borel $\sigma$-field is $\mathcal{F}$, see Proposition 7.1 in Chapter 3 of \cite{EthierKurz}. This in particular implies that $\mathcal{F}$ is separable, as the Borel $\sigma$-field of a separable metric space.

\begin{remark}\label{RApp}
Previous definitions and all the notions of this Appendix,
 extend to a time interval equal to $\mathbbm{R}_+$ or replacing the Skorokhod space with the Wiener space of continuous functions from $[0,T]$ (or $\mathbbm{R}_+$) to $E$.
\end{remark}

\begin{definition}\label{Defp}
The function 
\begin{equation*}
    p:\begin{array}{rcl}
        (s,x,t,A) &\longmapsto& p(s,x,t,A)   \\ \relax
        [0,T]\times E\times[0,T]\times\mathcal{B}(E) &\longrightarrow& [0,1], 
    \end{array}
\end{equation*}
will be called \textbf{transition function} if, for any $s,t$ in $[0,T]$, 
$x_0\in E$,  $A\in \mathcal{B}(E)$,
it verifies
\begin{enumerate}
\item $x \mapsto p(s,x,t,A)$ is Borel,
\item $B \mapsto p(s,x_0,t,B)$ is a probability measure on $(E,\mathcal{B}(E))$,
\item if $t\leq s$ then $p(s,x_0,t,A)=\mathds{1}_{A}(x_0)$,
\item if $s<t$, for any $u>t$, $\int_{E} p(s,x_0,t,dy)p(t,y,u,A) = 
p(s,x_0,u,A)$.
\end{enumerate}
\end{definition}
The latter statement is the well-known \textbf{Chapman-Kolmogorov equation}.

\begin{definition}\label{DefFoncTrans}
A transition function $p$ for which  the first item is reinforced 
supposing that $(s,x)\longmapsto p(s,x,t,A)$ is Borel for any $t,A$,
 will be said  \textbf{measurable in time}.
\end{definition}

\begin{definition}\label{defMarkov}
A \textbf{Markov canonical class} associated to a transition function $p$ is a set of probability measures $(\mathbbm{P}^{s,x})_{(s,x)\in[0,T]\times E}$ defined on the measurable space 
$(\Omega,\mathcal{F})$ and verifying for any $t \in [0,T]$ and $A\in\mathcal{B}(E)$
\begin{equation}\label{Markov1}
\mathbbm{P}^{s,x}(X_t\in A)=p(s,x,t,A),
\end{equation}
and for any $s\leq t\leq u$
\begin{equation}\label{Markov2}
\mathbbm{P}^{s,x}(X_u\in A|\mathcal{F}_t)=p(t,X_t,u,A)\quad \mathbbm{P}^{s,x}\text{ a.s.}
\end{equation}
\end{definition}
\begin{remark}\label{Rfuturefiltration}
	Formula 1.7 in Chapter 6 of \cite{dynkin1982markov} states
	that for any $F\in \mathcal{F}_{t,T}$ yields
	\begin{equation}\label{Markov3}
	\mathbbm{P}^{s,x}(F|\mathcal{F}_t) = \mathbbm{P}^{t,X_t}(F)=\mathbbm{P}^{s,x}(F|X_t)\,\text{  }\,  \mathbbm{P}^{s,x} \text{a.s.}
	\end{equation}
	Property  \eqref{Markov3}  will  be called 
	\textbf{Markov property}.
\end{remark}
For the rest of this section, we are given a Markov canonical  class $(\mathbbm{P}^{s,x})_{(s,x)\in[0,T]\times E}$ with transition function $p$.
\\
We will complete the $\sigma$-fields ${\mathcal F}_t$ of the canonical filtration by $\mathbbm{P}^{s,x}$ as follows.
\begin{definition}\label{CompletedBasis}
For any $(s,x)\in[0,T]\times E$ we will consider the  $(s,x)$-\textbf{completion} $\left(\Omega,\mathcal{F}^{s,x},(\mathcal{F}^{s,x}_t)_{t\in[0,T]},\mathbbm{P}^{s,x}\right)$ of the stochastic basis $\left(\Omega,\mathcal{F},(\mathcal{F}_t)_{t\in[0,T]},\mathbbm{P}^{s,x}\right)$ by defining $\mathcal{F}^{s,x}$ as  the $\mathbbm{P}^{s,x}$-completion of $\mathcal{F}$, by extending $\mathbbm{P}^{s,x}$ to $\mathcal{F}^{s,x}$ and finally by defining  $\mathcal{F}^{s,x}_t$ as the $\mathbbm{P}^{s,x}$-closure of $\mathcal{F}_t$ (meaning $\mathcal{F}_t$ augmented with the $\mathbbm{P}^{s,x}$-negligible sets) for every $t\in[0,T]$. 
\end{definition}

We remark that, for any $(s,x)\in[0,T]\times E$, $\left(\Omega,\mathcal{F}^{s,x},(\mathcal{F}^{s,x}_t)_{t\in[0,T]},\mathbbm{P}^{s,x}\right)$ is a stochastic basis fulfilling the usual conditions, see (1.4) in chapter I of \cite{jacod}).
We recall that considering a conditional expectation with respect to a $\sigma$-field augmented with the negligible sets or not, does not change the result. In particular we have the following.
\begin{proposition}\label{ConditionalExp}
Let $(\mathbbm{P}^{s,x})_{(s,x)\in[0,T]\times E}$ be a Markov canonical class. Let $(s,x)\in[0,T]\times E$ be fixed, $Z$ be a random variable and $t\in[s,T]$, then 
\\
$\mathbbm{E}^{s,x}[Z|\mathcal{F}_t]=\mathbbm{E}^{s,x}[Z|\mathcal{F}^{s,x}_t]$ $\mathbbm{P}^{s,x}$ a.s., provided that the left-hand (or the right-hand) side
is well-defined.
\end{proposition}
We state the following technical results of measurability without proofs. The interested reader can find complete proofs in \cite{paper2},
which are adapted from the time-homogeneous theory which the interested reader can find in \cite{dellmeyerD} for instance, see Proposition 10.a and Theorem 39 in its chapter XIV.

 \begin{proposition}\label{Borel}
 Let  $Z$ be a random variable. If the function $(s,x)\longmapsto \mathbbm{E}^{s,x}[Z]$ 
 is well-defined (with possible values in $[-\infty, \infty]$), then at fixed
 $s\in[0,T]$,
 $x\longmapsto \mathbbm{E}^{s,x}[Z]$ is Borel. 
 If moreover the transition function $p$ is measurable in time then, $(s,x)\longmapsto \mathbbm{E}^{s,x}[Z]$ is Borel.
 
In particular if $F\in \mathcal{F}$ be fixed, then at fixed
$s\in[0,T]$, $x\longmapsto \mathbbm{P}^{s,x}(F)$ is Borel. If 
 the transition function $p$ is measurable in time then, $(s,x)\longmapsto \mathbbm{P}^{s,x}(F)$ is Borel. 
 \end{proposition}

\begin{lemma}\label{LemmaBorel} 
Assume that the transition function of the Markov canonical class is measurable in time.

Let 
$f\in\mathcal{B}([0,T]\times E,\mathbbm{R})$ be such that for every $(s,x)\in[0,T]\times E$,\\
 $\mathbbm{E}^{s,x}[\int_s^{T}|f(r,X_r)|dr]<\infty$. Then 
$(s,x)\longmapsto \mathbbm{E}^{s,x}[\int_s^{T}f(r,X_r)dr]$ is Borel.
\end{lemma}

\begin{proposition}\label{measurableint}
  Let $f\in\mathcal{B}([0,T]\times E,\mathbbm{R})$ be such that for any $(s,x,t)$, $\mathbbm{E}^{s,x}[|f(t,X_t)|]<\infty$ then at fixed
  $s\in[0,T]$, $(x,t)\longmapsto \mathbbm{E}^{s,x}[f(t,X_t)]$ is Borel. If moreover the transition function $p$ is measurable in time, then 
\\
$(s,x,t)\longmapsto \mathbbm{E}^{s,x}[f(t,X_t)]$ is Borel.
\end{proposition}

\section{Technicalities concerning homogeneous Markov classes and martingale problems}\label{SC}

We start by introducing homogeneous Markov classes. In this section, we are given a Polish space $E$ and some $T\in\mathbbm{R}^*$.
\begin{notation}\label{HomogeneNonHomogene}
	A mapping $\tilde{p}:E\times[0,T]\times\mathcal{B}(E)$ will be called a \textbf{homogeneous transition function} if 
	$p:(s,x,t,A)\longmapsto \tilde{p}(x,t-s,A)\mathds{1}_{s<t}+\mathds{1}_A(x)\mathds{1}_{s\geq t}$ is a transition function in the sense of Definition \ref{Defp}. This in particular implies $\tilde{p}=p(0,\cdot,\cdot,\cdot)$.
	\\
	A set of probability measures $(\mathbbm{P}^x)_{x\in E}$ on the canonical space associated to $T$ and $E$ (see Notation \ref{canonicalspace}) will be called a \textbf{homogeneous Markov class} associated to a homogeneous transition function $\tilde{p}$ if 
\begin{equation}
\left\{\begin{array}{l}
 \forall t\in[0,T]\quad \forall A\in\mathcal{B}(E)\quad ,\mathbbm{P}^{x}(X_t\in A)=\tilde{p}(x,t,A)\\
 \forall 0\leq t\leq u\leq T\quad , \mathbbm{P}^x(X_u\in A|\mathcal{F}_t)=
\tilde p(X_t,u-t,A)\quad \mathbbm{P}^{s,x}  \text{a.s.}
\end{array}\right.
\end{equation}	
Given a homogeneous Markov class $(\mathbbm{P}^x)_{x\in E}$ associated to a homogeneous transition function $\tilde{p}$, one can always consider the Markov class $(\mathbbm{P}^{s,x})_{(s,x)\in[0,T]\times E}$ associated to the transition function 
	$p:(s,x,t,A)\longmapsto \tilde{p}(x,t-s,A)\mathds{1}_{s<t}+\mathds{1}_A(x)\mathds{1}_{s\geq t}$. 
In particular, for any $x\in E$, we have $\mathbbm{P}^{0,x}=\mathbbm{P}^x$.

\end{notation}
We show that a homogeneous transition function necessarily produces a measurable in time non homogeneous transition function.

\begin{proposition}\label{HomoMeasurableint}
Let $\tilde{p}$ be a homogeneous transition function and let $p$ be the associated non homogeneous transition function as described in Notation \ref{HomogeneNonHomogene}. Then $p$ is measurable in time in the sense of Definition  \ref{DefFoncTrans}.

\end{proposition}

\begin{proof}
Given that $p:(s,x,t,A)\longmapsto \tilde{p}(x,t-s,A)\mathds{1}_{s<t}+\mathds{1}_A(x)\mathds{1}_{s\geq t}$, it is actually enough to show that $\tilde{p}(\cdot,\cdot,A)$ is Borel for any $A\in\mathcal{B}(E)$. We can also write $ \tilde{p}=p(0,\cdot,\cdot,\cdot)$, so $p$ is measurable in time if $p(0,\cdot,\cdot,A)$ is Borel for any $A\in\mathcal{B}(E)$, and this holds thanks to Proposition \ref{measurableint} applied to $f:=\mathds{1}_A$ .
\end{proof}

We then introduce below the notion of  homogeneous martingale problems.
 \begin{definition}\label{MPhomogene}
 	 Given $A$ an operator mapping a linear algebra  $\mathcal{D}(A)\subset \mathcal{B}_b(E,\mathbbm{R})$ into
         $\mathcal{B}_b(E,\mathbbm{R})$, we say that a set of probability
         measures $(\mathbbm{P}^x)_{x\in E}$ on the measurable space $(\Omega,\mathcal{F})$
   (see Notation \ref{canonicalspace}) solves the
   \textbf{homogeneous Martingale Problem} associated to $(\mathcal{D}(A),A)$ if for any $x\in E$, $\mathbbm{P}^x$ satisfies
 	\begin{itemize}
 		\item for every $\phi\in\mathcal{D}(A)$, $\phi(X_{\cdot})-\int_0^{\cdot}A\phi(X_r)dr$ is a $(\mathbbm{P}^x,(\mathcal{F}_t)_{t\in[0,T]})$-local
 martingale;
\item $\mathbbm{P}^x(X_0=x)=1$.
 	\end{itemize}
 We say that this {\bf homogeneous Martingale Problem is well-posed} if for any $x\in E$, $\mathbbm{P}^x$ is the only probability measure on $(\Omega,\mathcal{F})$ verifying those two items.
 \end{definition}

\begin{remark}
	 If $(\mathbbm{P}^x)_{x\in E}$ is a homogeneous Markov class solving the homogeneous Martingale Problem associated to some $(\mathcal{D}(A),A)$, then the corresponding $(\mathbbm{P}^{s,x})_{(s,x)\in[0,T]\times E}$ (see Notation \ref{HomogeneNonHomogene}) solves the Martingale Problem associated to $(\mathcal{D}(A),A)$ in the sense of Definition \ref{MartingaleProblem}. Moreover if the homogeneous Martingale Problem is well-posed, so is the latter one.
\end{remark}
So a homogeneous Markov process solving a homogeneous martingale problem falls
 into our setup. We will now see how we can pass from an operator $A$ which only
 acts on time-independent functions to an evolution operator $\partial_t+A$, 
and see how our Markov class still solves the corresponding martingale problem.

\begin{notation}\label{NotDomain}
	Let $E$ be a Polish space and let $A$ be an operator mapping a linear algebra  $\mathcal{D}(A)\subset \mathcal{B}_b(E,\mathbbm{R})$ into $\mathcal{B}_b(E,\mathbbm{R})$.
	\\
	If $\phi\in\mathcal{B}([0,T]\times E,\mathbbm{R})$ is such that for every $t\in[0,T]$, $\phi(t,\cdot)\in\mathcal{D}(A)$, then  $A\phi$ will denote the
 mapping $(t,x)\longmapsto A(\phi(t,\cdot))(x)$.
	\\
	\\
	We now introduce the time-inhomogeneous domain associated to $A$ which we denote $\mathcal{D}^{max}(\partial_t+A)$ and which consists in functions $\phi\in\mathcal{B}_b([0,T]\times E,\mathbbm{R})$ verifying the following conditions:
	\begin{itemize}
		\item $\forall x\in E$, $\phi(\cdot,x)\in\mathcal{C}^1([0,T],\mathbbm{R})$ and $\forall t\in[0,T]$, $\phi(t,\cdot)\in\mathcal{D}(A)$;
		\item $\forall t\in[0,T]$, $\partial_t \phi(t,\cdot)\in\mathcal{D}(A)$ and $\forall x\in E$, $A\phi(\cdot,x)\in\mathcal{C}^1([0,T],\mathbbm{R})$;
		\item $\partial_t\circ A\phi=A\circ\partial_t\phi$;
		\item $\partial_t\phi$, $A\phi$ and $\partial_t\circ A\phi$ belong to $\mathcal{B}_b([0,T]\times E,\mathbbm{R})$.
	\end{itemize}
	On $\mathcal{D}^{max}(\partial_t+A)$ we will consider the operator $\partial_t+A$.
\end{notation}

\begin{remark}\label{RemDomain}
	With these notations, it is clear that $\mathcal{D}^{max}(\partial_t+A)$ is a sub-linear space of $\mathcal{B}_b([0,T]\times E,\mathbbm{R})$. It is in general not a linear algebra, but always  contains $\mathcal{D}(A)$, and even $\mathcal{C}^1([0,T],\mathbbm{R})\otimes\mathcal{D}(A)$, the linear algebra of functions which can be written $\underset{k\leq N}{\sum}\lambda_k\psi_k\phi_k$ where $N\in\mathbbm{N}$, and for any $k$, $\lambda_k\in\mathbbm{R}$, $\psi_k\in\mathcal{C}^1([0,T],\mathbbm{R})$, $\phi_k\in\mathcal{D}(A)$.	We also notice that $\partial_t+A$ maps $\mathcal{D}^{max}(\partial_t+A)$ into $\mathcal{B}_b([0,T]\times E,\mathbbm{R})$.
\end{remark}

\begin{lemma}\label{LemDomain}
	Let us consider the same notations and under the same assumptions as 
	in Notation \ref{NotDomain}. Let $(\mathbbm{P}^{s,x})_{(s,x)\in[0,T]\times E}$ be a Markov class solving the well-posed Martingale Problem associated to $(A,\mathcal{D}(A))$ in the sense of Definition \ref{MartingaleProblem}. Then it also solves the well-posed martingale problem associated to
	$(\partial_t+A,\mathcal{A})$ for any linear algebra $\mathcal{A}$ included in $\mathcal{D}^{max}(\partial_t+A)$.
\end{lemma}

\begin{proof}
	We start by noticing that since $\mathcal{D}(A)\subset \mathcal{B}_b(E,\mathbbm{R})$ and is mapped into  $\mathcal{B}_b(E,\mathbbm{R})$, then for any $(s,x)\in[0,T]\times E$ and $\phi\in\mathcal{D}(A)$, $M^{s,x}[\phi]$ is bounded and is therefore a martingale. 
	\\
	\\
	We fix $(s,x)\in[0,T]\times E$, $\phi\in\mathcal{D}^{max}(\partial_t+A)$ and $s\leq t\leq u\leq T$ and we will show that
	\begin{equation}\label{eqDomain}
	\mathbbm{E}^{s,x}\left[\phi(u,X_u)-\phi(t,X_t)-\int_t^u(\partial_t+A)\phi(r,X_r)dr\middle|\mathcal{F}_t\right]=0,
	\end{equation}
which implies	 that $\phi(\cdot,X_{\cdot})-\int_s^{\cdot}(\partial_t+A)\phi(r,X_r)dr$, $t\in[s,T]$ is a $\mathbbm{P}^{s,x}$-martingale.
	We have 
	\begin{equation*}
	\begin{array}{rcl}
	&&\mathbbm{E}^{s,x}[\phi(u,X_u)-\phi(t,X_t)|\mathcal{F}_t]\\
	&=&\mathbbm{E}^{s,x}[(\phi(u,X_t)-\phi(t,X_t))+(\phi(u,X_u)-\phi(u,X_t))|\mathcal{F}_t]\\
	&=&\mathbbm{E}^{s,x}\left[\int_t^u\partial_t\phi(r,X_t)dr+\left(\int_t^uA\phi(u,X_r)dr+(M^{s,x}[\phi(u,\cdot)]_u-M^{s,x}[\phi(u,\cdot)]_t)\right)|\mathcal{F}_t\right]\\
	&=&\mathbbm{E}^{s,x}\left[\int_t^u\partial_t\phi(r,X_t)dr+\int_t^uA\phi(u,X_r)dr|\mathcal{F}_t\right] \\
&=& I_0 - I_1 + I_2,
	\end{array}
	\end{equation*} 
        where
        \begin{eqnarray*}
          I_0  &=& \mathbbm{E}^{s,x}\left[\int_t^u\partial_t\phi(r,X_r)dr+\int_t^uA\phi(r,X_r)dr|\mathcal{F}_t\right]\\
          I_1 &=&\mathbbm{E}^{s,x}\left[\int_t^u(\partial_t\phi(r,X_r)-\partial_t\phi(r,X_t))dr|\mathcal{F}_t\right]\\
          I_2 &=&\mathbbm{E}^{s,x}\left[\int_t^u(A\phi(u,X_r)-A\phi(r,X_r))dr|\mathcal{F}_t\right].
\end{eqnarray*}
                  \eqref{eqDomain} will be established 
if  one proves that $I_1 = I_2$. We do this below. 

At fixed
$r$ and $\omega$, $v\longmapsto A\phi(v,X_r(\omega))$ is $\mathcal{C}^1$, therefore 
	$A\phi(u,X_r(\omega))-A\phi(r,X_r(\omega))=\int_r^u\partial_tA\phi(v,X_r(\omega))dv$ and 
	$I_2= \mathbbm{E}^{s,x}\left[\int_t^u\int_r^u\partial_tA\phi(v,X_r)dvdr|\mathcal{F}_t\right]$. Then
\begin{eqnarray*}
        I_1&=&\mathbbm{E}^{s,x}\left[\int_t^u\int_t^rA\partial_t\phi(r,X_v)dvdr|\mathcal{F}_t\right]\\ &+&\mathbbm{E}^{s,x}
               \left[\int_t^u(M^{s,x}[\partial_t\phi(r,\cdot)]_r
  - M^{s,x}[\partial_t\phi(r,\cdot)]_t)dr|\mathcal{F}_t \right].\end{eqnarray*}
	Since $\partial_t\phi$ and $A\partial_t\phi$ are bounded, $M^{s,x}[\partial_t\phi(r,\cdot)]_r(\omega)$ is uniformly bounded in $(r,\omega)$, so by Fubini's theorem for conditional expectations we have
	\begin{equation}
	\begin{array}{rcl}
	&&\mathbbm{E}^{s,x}[\int_t^u(M^{s,x}[\partial_t\phi(r,\cdot)]_r-M^{s,x}[\partial_t\phi(r,\cdot)]_t)dr|\mathcal{F}_t]\\
	&=&\int_t^u\mathbbm{E}^{s,x}[M^{s,x}[\partial_t\phi(r,\cdot)]_r-M^{s,x}[\partial_t\phi(r,\cdot)]_t|\mathcal{F}_t]dr\\
	&=&0.
	\end{array}
	\end{equation}  
	Finally since $\partial_tA\phi=A\partial_t\phi$ and again  by Fubini's theorem for conditional expectations, we have $\mathbbm{E}^{s,x}\left[\int_t^u\int_r^u\partial_tA\phi(v,X_r)dvdr|\mathcal{F}_t\right]=\mathbbm{E}^{s,x}\left[\int_t^u\int_t^rA\partial_t\phi(r,X_v)dvdr|\mathcal{F}_t\right]$ so $I_1=I_2$ which concludes the proof.
\end{proof}

In conclusion we can state the following.
\begin{corollary}\label{conclusionA4}
Given a homogeneous Markov class $(\mathbbm{P}^x)_{x\in E}$ solving a well-posed homogeneous Martingale Problem associated to some $(\mathcal{D}(A),A)$, 
there exists  a Markov class $(\mathbbm{P}^{s,x})_{(s,x)\in[0,T]\times E}$ which transition function is measurable in time and  such that for any algebra $\mathcal{A}$ included in $\mathcal{D}^{max}(\partial_t+A)$,  $(\mathbbm{P}^{s,x})_{(s,x)\in[0,T]\times E}$ solves the well-posed Martingale Problem associated to $(\partial_t+A,\mathcal{A})$ in the sense of Definition \ref{MartingaleProblem}.
\end{corollary}

\end{appendices}
\section*{Acknowledgments}

The authors are grateful to the Referee for the careful reading
and the stimulating comments.
The work of the second named author
was partially supported by a public grant as part of the
{\it Investissement d'avenir project, reference ANR-11-LABX-0056-LMH,
  LabEx LMH,}
in a joint call with Gaspard Monge Program for optimization, operations research and their interactions with data sciences.

\bibliographystyle{plain}
\bibliography{../../../biblioPhDBarrasso_bib/biblioPhDBarrasso}

\def\polhk#1{\setbox0=\hbox{#1}{\ooalign{\hidewidth
  \lower1.5ex\hbox{`}\hidewidth\crcr\unhbox0}}}
  \def\polhk#1{\setbox0=\hbox{#1}{\ooalign{\hidewidth
  \lower1.5ex\hbox{`}\hidewidth\crcr\unhbox0}}} \def\cprime{$'$}
  \def\polhk#1{\setbox0=\hbox{#1}{\ooalign{\hidewidth
  \lower1.5ex\hbox{`}\hidewidth\crcr\unhbox0}}}
\begin{thebibliography}{10}

\bibitem{ar}
D.~G. Aronson.
\newblock Bounds for the fundamental solution of a parabolic equation.
\newblock {\em Bull. Amer. Math. Soc.}, 73:890--896, 1967.

\bibitem{bakry}
D.~Bakry, I.~Gentil, and M.~Ledoux.
\newblock {\em Analysis and geometry of Markov diffusion operators}, volume
  348.
\newblock Springer Science \& Business Media, 2013.

\bibitem{BSDEmildPardouxBally}
V.~Bally, E.~Pardoux, and L.~Stoica.
\newblock Backward stochastic differential equations associated to a symmetric
  {M}arkov process.
\newblock {\em Potential Anal.}, 22(1):17--60, 2005.

\bibitem{barles1997backward}
G.~Barles, R.~Buckdahn, and E.~Pardoux.
\newblock Backward stochastic differential equations and integral-partial
  differential equations.
\newblock {\em Stochastics: An International Journal of Probability and
  Stochastic Processes}, 60(1-2):57--83, 1997.

\bibitem{barles1997sde}
G.~Barles and E.~Lesigne.
\newblock S{DE}, {BSDE} and {PDE}.
\newblock In {\em Backward stochastic differential equations ({P}aris,
  1995--1996)}, volume 364 of {\em Pitman Res. Notes Math. Ser.}, pages 47--80.
  Longman, Harlow, 1997.

\bibitem{paper1preprint}
A.~Barrasso and F.~Russo.
\newblock {Backward Stochastic Differential Equations with no driving
  martingale, Markov processes and associated Pseudo Partial Differential
  Equations}.
\newblock Preprint, hal-01431559, December 2017.

\bibitem{paper2}
A.~Barrasso and F.~Russo.
\newblock {BSDEs with no driving martingale, Markov processes and associated
  Pseudo Partial Differential Equations. Part II: Decoupled mild solutions and
  Examples.}
\newblock Preprint hal-01505974, v3, 2020.

\bibitem{bismut}
J.M. Bismut.
\newblock Conjugate convex functions in optimal stochastic control.
\newblock {\em J. Math. Anal. Appl.}, 44:384--404, 1973.

\bibitem{daprato_zabczyk14}
G.~Da~Prato and J.~Zabczyk.
\newblock {\em Stochastic equations in infinite dimensions}, volume 152 of {\em
  Encyclopedia of Mathematics and its Applications}.
\newblock Cambridge University Press, Cambridge, second edition, 2014.

\bibitem{diel}
F.~Delarue and R.~Diel.
\newblock Rough paths and 1d {SDE} with a time dependent distributional drift:
  application to polymers.
\newblock {\em Probab. Theory Related Fields}, 165(1-2):1--63, 2016.

\bibitem{dellmeyerD}
C.~Dellacherie and P.-A. Meyer.
\newblock {\em Probabilit\'es et potentiel. {C}hapitres {XII}--{XVI}}.
\newblock Publications de l'Institut de Math\'ematiques de l'Universit\'e de
  Strasbourg [Publications of the Mathematical Institute of the University of
  Strasbourg], XIX. Hermann, Paris, second edition, 1987.
\newblock Th{\'e}orie des processus de Markov. [Theory of Markov processes].

\bibitem{di2012hitchhikers}
E.~Di~Nezza, G.~Palatucci, and E.~Valdinoci.
\newblock Hitchhiker's guide to the fractional {S}obolev spaces.
\newblock {\em Bull. Sci. Math.}, 136(5):521--573, 2012.

\bibitem{dynkin1982markov}
E.~B. Dynkin.
\newblock {\em Markov processes and related problems of analysis}, volume~54 of
  {\em London Mathematical Society Lecture Note Series}.
\newblock Cambridge University Press, Cambridge-New York, 1982.

\bibitem{el1997backward}
N.~El~Karoui, S.~Peng, and M.~C. Quenez.
\newblock Backward stochastic differential equations in finance.
\newblock {\em Mathematical finance}, 7(1):1--71, 1997.

\bibitem{EthierKurz}
S.~N. Ethier and T.~G. Kurtz.
\newblock {\em Markov processes}.
\newblock Wiley Series in Probability and Mathematical Statistics: Probability
  and Mathematical Statistics. John Wiley \& Sons, Inc., New York, 1986.
\newblock Characterization and convergence.

\bibitem{issoglio}
F.~Flandoli, E.~Issoglio, and F.~Russo.
\newblock Multidimensional stochastic differential equations with
  distributional drift.
\newblock {\em Trans. Amer. Math. Soc.}, 369(3):1665--1688, 2017.

\bibitem{frw1}
F.~Flandoli, F.~Russo, and J.~Wolf.
\newblock Some {SDE}s with distributional drift. {I}. {G}eneral calculus.
\newblock {\em Osaka J. Math.}, 40(2):493--542, 2003.

\bibitem{frw2}
F.~Flandoli, F.~Russo, and J.~Wolf.
\newblock Some {SDE}s with distributional drift. {II}. {L}yons-{Z}heng
  structure, {I}t\^o's formula and semimartingale characterization.
\newblock {\em Random Oper. Stochastic Equations}, 12(2):145--184, 2004.

\bibitem{fuhrman_tessitore02}
M.~Fuhrman and G.~Tessitore.
\newblock Nonlinear {K}olmogorov equations in infinite dimensional spaces: the
  backward stochastic differential equations approach and applications to
  optimal control.
\newblock {\em Ann. Probab.}, 30(3):1397--1465, 2002.

\bibitem{fuosta}
M.~Fukushima, Y.~Oshima, and M.~Takeda.
\newblock {\em Dirichlet forms and symmetric Markov processes}, volume~19 of
  {\em de Gruyter Studies in Mathematics}.
\newblock Walter de Gruyter \& Co., Berlin, 1994.

\bibitem{hoh}
W.~Hoh.
\newblock Pseudo differential operators generating markov processes.
\newblock {\em Habilitations-schrift, Universit{\"a}t Bielefeld}, 1998.

\bibitem{hsu}
E.~P. Hsu.
\newblock {\em Stochastic analysis on manifolds}, volume~38.
\newblock American Mathematical Soc., 2002.

\bibitem{issoglio_jing16}
E.~Issoglio and S.~Jing.
\newblock Forward-backward {SDE}s with distributional coefficients.
\newblock {\em Stochastic Process. Appl.}, 130(1):47--78, 2020.

\bibitem{jacob1}
N.~Jacob.
\newblock {\em Pseudo differential operators and {M}arkov processes. {V}ol.
  {I}}.
\newblock Imperial College Press, London, 2001.
\newblock Fourier analysis and semigroups.

\bibitem{jacob2}
N.~Jacob.
\newblock {\em Pseudo differential operators \& {M}arkov processes. {V}ol.
  {II}}.
\newblock Imperial College Press, London, 2002.
\newblock Generators and their potential theory.

\bibitem{jacob2005pseudo}
N.~Jacob.
\newblock {\em Pseudo Differential Operators \& Markov Processes: Markov
  Processes And Applications}, volume~3.
\newblock Imperial College Press, 2005.

\bibitem{jacod79}
J.~Jacod.
\newblock {\em Calcul stochastique et probl\`emes de martingales}, volume 714
  of {\em Lecture Notes in Mathematics}.
\newblock Springer, Berlin, 1979.

\bibitem{jacod}
J.~Jacod and A.~N. Shiryaev.
\newblock {\em Limit theorems for stochastic processes}, volume 288 of {\em
  Grundlehren der Mathematischen Wissenschaften [Fundamental Principles of
  Mathematical Sciences]}.
\newblock Springer-Verlag, Berlin, second edition, 2003.

\bibitem{jost}
J.~Jost.
\newblock {\em Riemannian Geometry and Geometric Analysis}.
\newblock Berlin, Heidelberg: Springer Berlin Heidelberg, 2011.

\bibitem{klimsiak}
T.~Klimsiak.
\newblock Semi-{D}irichlet forms, {F}eynman-{K}ac functionals and the {C}auchy
  problem for semilinear parabolic equations.
\newblock {\em J. Funct. Anal.}, 268(5):1205--1240, 2015.

\bibitem{qianBSDEs}
G.~Liang, T.~Lyons, and Z.~Qian.
\newblock Backward stochastic dynamics on a filtered probability space.
\newblock {\em Ann. Probab.}, 39(4):1422--1448, 2011.

\bibitem{sem10}
P.A. Meyer.
\newblock {\em S\'eminaire de {P}robabilit\'es, {X}}.
\newblock Lecture Notes in Mathematics, Vol. 511. Springer-Verlag, Berlin-New
  York, 1976.
\newblock Tenu {\`a} l'Universit{\'e} de Strasbourg, Strasbourg. Premi{\`e}re
  partie (ann{\'e}e universitaire 1974/1975). Seconde partie: Th{\'e}orie des
  int{\'e}grales stochastiques (ann{\'e}e universitaire 1974/1975). Expos{\'e}s
  suppl{\'e}mentaires, Edit{\'e} par P. A. Meyer.

\bibitem{parpen90}
{\'E}.~Pardoux and S.~Peng.
\newblock Adapted solution of a backward stochastic differential equation.
\newblock {\em Systems Control Lett.}, 14(1):55--61, 1990.

\bibitem{pardoux1992backward}
{\'E}.~Pardoux and S.~Peng.
\newblock Backward stochastic differential equations and quasilinear parabolic
  partial differential equations.
\newblock In {\em Stochastic partial differential equations and their
  applications ({C}harlotte, {NC}, 1991)}, volume 176 of {\em Lecture Notes in
  Control and Inform. Sci.}, pages 200--217. Springer, Berlin, 1992.

\bibitem{PardouxRascanu}
E.~Pardoux and A.~R{\u{a}}{\c{s}}canu.
\newblock {\em Stochastic differential equations, backward {SDE}s, partial
  differential equations}, volume~69 of {\em Stochastic Modelling and Applied
  Probability}.
\newblock Springer, Cham, 2014.

\bibitem{peng1991probabilistic}
S.~Peng.
\newblock Probabilistic interpretation for systems of quasilinear parabolic
  partial differential equations.
\newblock {\em Stochastics Stochastics Rep.}, 37(1-2):61--74, 1991.

\bibitem{roth}
J.P. Roth.
\newblock Op\'erateurs dissipatifs et semi-groupes dans les espaces de
  fonctions continues.
\newblock {\em Ann. Inst. Fourier (Grenoble)}, 26(4):ix, 1--97, 1976.

\bibitem{russo_trutnau07}
F.~Russo and G.~Trutnau.
\newblock Some parabolic {PDE}s whose drift is an irregular random noise in
  space.
\newblock {\em Ann. Probab.}, 35(6):2213--2262, 2007.

\bibitem{wurzer}
F.~Russo and L.~Wurzer.
\newblock Elliptic {PDE}s with distributional drift and backward {SDE}s driven
  by a c\`adl\`ag martingale with random terminal time.
\newblock {\em Stoch. Dyn.}, 17(4):1750030, 36, 2017.

\bibitem{stroock1975diffusion}
D.~W. Stroock.
\newblock Diffusion processes associated with {L}\'evy generators.
\newblock {\em Z. Wahrscheinlichkeitstheorie und Verw. Gebiete},
  32(3):209--244, 1975.

\bibitem{stroock}
D.~W. Stroock and S.~R.~S. Varadhan.
\newblock {\em Multidimensional diffusion processes}.
\newblock Classics in Mathematics. Springer-Verlag, Berlin, 2006.
\newblock Reprint of the 1997 edition.

\end{thebibliography}
\end{document}